\documentclass[a4paper]{article}

\usepackage[utf8]{inputenc}
\usepackage[T1]{fontenc}
\usepackage{amsmath}
\usepackage{amsthm}
\usepackage{amsfonts}
\usepackage{mathtools}
\usepackage{xypic}
\usepackage{tikz}
\usepackage{tikz-cd}
\usepackage{subcaption}
\usepackage{float}
\usepackage{natbib}
\usepackage{import}
\usepackage{hyperref}
\usepackage{import} 

\newtheorem{theorem}{Theorem}
\newtheorem{definition}{Definition}
\newtheorem{lemma}{Lemma}

\DeclareMathOperator{\Ob}{Ob}
\DeclareMathOperator{\Mor}{Mor}
\DeclareMathOperator{\dom}{dom}
\DeclareMathOperator{\cod}{cod}

\newcommand{\C}{\mathcal{C}}
\newcommand{\D}{\mathcal{D}}

\usepackage{authblk}

\title{Sheet diagrams for bimonoidal categories}
\author[1]{Cole Comfort}
\author[1]{Antonin Delpeuch}
\author[2]{Jules Hedges}
\affil[1]{Department of Computer Science, University of Oxford}
\affil[2]{Department of Computer and Information Sciences, University of Strathclyde}

\begin{document}
\maketitle

\begin{abstract}
	Bimonoidal categories (also known as rig categories) are categories with two monoidal structures, one of which distributes over the other. We formally define sheet diagrams, a graphical calculus for bimonoidal categories that was informally
introduced by Staton. Sheet diagrams are string diagrams drawn on a branching surface, which is itself an extruded string diagram. Our main result is a soundness and completeness theorem of the usual form for graphical calculi: we show that sheet diagrams
form the free bimonoidal category on a signature.
\end{abstract}

\section{Introduction}

String diagrams for monoidal categories are graphical representations of morphisms in monoidal categories ~\citep{hotz1965algebraisierung-1,joyal1991geometry-1}.    The equational theory for monoidal categories corresponds to certain permissible
topological transformations from one diagram to another,  so that one diagram can be deformed into another if and only if they are equal up to the axioms of monoidal categories.  Thus one can reason about monoidal categories graphically, by deforming the corresponding string diagrams.

Monoidal categories are common in applied settings, usually as a minimal theory of processes, where the categorical composition and monoidal product of morphisms are respectively thought of as sequential and parallel composition of processes (see for example \citep{fong_spivak_seven_sketches,coecke_kissinger_picturing_quantum_processes}). String diagrams have the simultaneous advantages that they are so intuitive that beginners can pick them up without even having knowledge of category theory, while also being entirely formal thanks to the soundness and completeness theorems which relate deformations to axiomatic reasoning with algebraic expressions.\footnote{Such results are sometimes called \emph{coherence theorems}, but we reserve this term to results about commutation of all diagrams of a certain sort.} String diagrams have been defined for several more refined notions of monoidal categories, a survey of which can be found in~\cite{selinger2010survey-1}.

However convenient this syntax might be, it is limited by the expressivity of the monoidal structure. 
It is common to work in settings where two monoidal structures are used; however, monoidal string diagrams are by definition a syntax for one particular monoidal structure.  And a priori, there is no reason why one might expect that more than one monoidal structures interact well with each other.  However, there are certain instances when there is some sort of useful interaction between both monoidal structures.

For the purpose of this paper we focus our attention to a particular type of distributivity between two monoidal structures on a category.
Bimonoidal categories (also known as rig categories) are categories with a monoidal structure $\otimes$ and a symmetric monoidal structure $\oplus$,  with natural  isomorphisms $$\delta_{A,B,C}:A \otimes(B\oplus C) \to (A\otimes B) \oplus (A\otimes C)$$
$$\delta_{A,B,C}^\#: (A\oplus B) \otimes C \to (A \otimes C)\oplus (B\otimes C)$$ called \emph{distributors}, distributing $\otimes$ over $\oplus$ from the left and the right, satisfying certain coherence laws. Many well-known categories have such a structure, for example $\mathbf{Set}$ with disjoint unions and cartesian products, or $\mathbf{Vect}$ with direct sums and tensor products. 
Some informal attempts have been made at drawing string diagrams for such categories.  Consider for instance the following linear map, a morphism in $\mathbf{Vect}$:
$$A \xrightarrow{f} (B \otimes C) \oplus (B \otimes C) \xrightarrow{1_{B \otimes B} \oplus \gamma_{B,C}} (B \otimes C) \oplus (C \otimes B) \xrightarrow{g} D \otimes E$$
where $\gamma_{B,C} : B \otimes C \rightarrow C \otimes B$ is the symmetry for $\otimes$.
Authors have used various informal conventions to represent such a morphism as a diagram:
\begin{figure}[H]
  \centering
  \begin{subfigure}{0.32\textwidth}
    \centering
    \begin{tikzpicture}
  \node[rectangle,draw] at (0,0) (f) {$f$};
  \node[rectangle,draw] at (0,2) (g) {$g$};
  
  \node at (0,-.75) (A) {$A$};
  \node at (-.4,2.75) (D) {$D$};
  \node at (.4,2.75) (E) {$E$};

  \draw (g) -- (D);
  \draw (g) -- (E);
  \draw (f) -- (A);

  \node at (-1,0) (p0) {};
  \node at (0,0) (pmid) {};
  \node at (1,0) (pend) {};

  \node at (-.6,0) (p1) {};
  \node at (-.4,0) (p2) {};
  \node at (.4,0) (p3) {};
  \node at (.6,0) (p4) {};
  \draw[gray,dashed] (p0 |- f.north) rectangle (pmid |- g.south);
  \draw[gray,dashed] (pmid |- g.south) -- (pend |- g.south) -- (pend |- f.north) -- (pmid |- f.north);
  
  \draw (p1 |- f.north) -- (p1 |- g.south);
  \draw (p2 |- f.north) -- (p2 |- g.south);

  \draw (p3 |- f.north) -- (p4 |- g.south);
  \draw (p4 |- f.north) -- (p3 |- g.south);

\end{tikzpicture}
    \caption{\cite{duncan2009generalized}}
  \end{subfigure}
  \begin{subfigure}{0.32\textwidth}
    \centering
    \begin{tikzpicture}
  \node[rectangle,draw,minimum width=1.5cm] at (0,0) (f) {$f$};
  \node[rectangle,draw,minimum width=1.5cm] at (0,2) (g) {$g$};
  
  \node at (0,-.75) (A) {$A$};
  \node at (-.4,2.75) (D) {$D$};
  \node at (.4,2.75) (E) {$E$};

  \draw (g) -- (D);
  \draw (g) -- (E);
  \draw (f) -- (A);

  \node at (0,1) (p) {$+$};
  \node at (-.6,0) (p1) {};
  \node at (-.4,0) (p2) {};
  \node at (.4,0) (p3) {};
  \node at (.6,0) (p4) {};
  \draw (p1 |- f.north) -- (p1 |- g.south);
  \draw (p2 |- f.north) -- (p2 |- g.south);

  \draw (p3 |- f.north) -- (p4 |- g.south);
  \draw (p4 |- f.north) -- (p3 |- g.south);

\end{tikzpicture}
    \caption{\cite{james2012information}}
  \end{subfigure}
  \begin{subfigure}{0.32\textwidth}
    \centering
    \def\svgscale{0.65}
    \small
\begingroup%
  \makeatletter%
  \providecommand\color[2][]{%
    \errmessage{(Inkscape) Color is used for the text in Inkscape, but the package 'color.sty' is not loaded}%
    \renewcommand\color[2][]{}%
  }%
  \providecommand\transparent[1]{%
    \errmessage{(Inkscape) Transparency is used (non-zero) for the text in Inkscape, but the package 'transparent.sty' is not loaded}%
    \renewcommand\transparent[1]{}%
  }%
  \providecommand\rotatebox[2]{#2}%
  \newcommand*\fsize{\dimexpr\f@size pt\relax}%
  \newcommand*\lineheight[1]{\fontsize{\fsize}{#1\fsize}\selectfont}%
  \ifx\svgwidth\undefined%
    \setlength{\unitlength}{73.94453384bp}%
    \ifx\svgscale\undefined%
      \relax%
    \else%
      \setlength{\unitlength}{\unitlength * \real{\svgscale}}%
    \fi%
  \else%
    \setlength{\unitlength}{\svgwidth}%
  \fi%
  \global\let\svgwidth\undefined%
  \global\let\svgscale\undefined%
  \makeatother%
  \begin{picture}(1,2.22659511)%
    \lineheight{1}%
    \setlength\tabcolsep{0pt}%
    \put(0,0){\includegraphics[width=\unitlength,page=1]{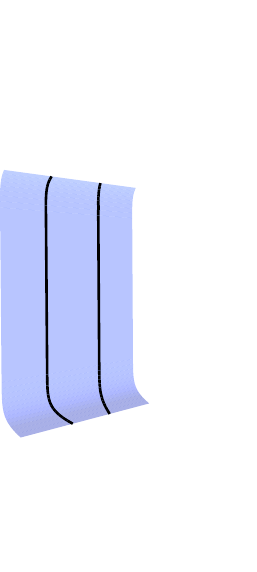}}%
    \put(0.62109997,2.05919014){\color[rgb]{0,0,0}\makebox(0,0)[t]{\lineheight{1.25}\smash{\begin{tabular}[t]{c}$E$\end{tabular}}}}%
    \put(0,0){\includegraphics[width=\unitlength,page=2]{intro.pdf}}%
    \put(0.53988178,0.01030122){\color[rgb]{0,0,0}\makebox(0,0)[t]{\lineheight{1.25}\smash{\begin{tabular}[t]{c}$A$\end{tabular}}}}%
    \put(0,0){\includegraphics[width=\unitlength,page=3]{intro.pdf}}%
    \put(0.53951614,0.58312905){\color[rgb]{0,0,0}\makebox(0,0)[t]{\lineheight{1.25}\smash{\begin{tabular}[t]{c}$f$\end{tabular}}}}%
    \put(0,0){\includegraphics[width=\unitlength,page=4]{intro.pdf}}%
    \put(0.53272832,1.84326998){\color[rgb]{0,0,0}\makebox(0,0)[t]{\lineheight{1.25}\smash{\begin{tabular}[t]{c}$g$\end{tabular}}}}%
    \put(0,0){\includegraphics[width=\unitlength,page=5]{intro.pdf}}%
    \put(0.43140699,2.14279866){\color[rgb]{0,0,0}\makebox(0,0)[t]{\lineheight{1.25}\smash{\begin{tabular}[t]{c}$D$\end{tabular}}}}%
  \end{picture}%
\endgroup%

    \caption{\cite{staton2015algebraic}}
  \end{subfigure}
\end{figure}
These conventions all communicate the structure of a morphism to readers, but do not a priori enjoy a soundness and completeness theorem.
In this paper we develop the formal theory of diagrams for bimonoidal categories (also known as rig categories, semiring categories). We provide a definition of the class of diagrams (that we call \emph{sheet diagrams}), their deformations and a soundness and completeness theorem for them. Our sheet diagrams follow the three-dimensional style used by \cite{staton2015algebraic} and \cite{delpeuch2019complete-1}, retrospectively justifying their use as formal reasoning tools.

Sheet diagrams represent morphisms in a bimonoidal category in a normal form, as a ``sum of products'' with $\otimes$ pushed to the inside and $\oplus$ pushed to the outside by the distributors, similar to the normal form of a polynomial, or disjunctive normal form in logic. This is intended as a compromise, making sheet diagrams both (in the authors' opinions) easier to visualise and also easier to typeset, in return for which the tensor product of sheet diagrams is a rather complicated derived operation.

Bimonoidal categories have found applications in a variety of fields:
probability theory~\citep{fritz2018bimonoidal}, quantum
information~\citep{staton2015algebraic}, dataflow
computations~\citep{delpeuch2019complete-1}, game theory~\citep{hedges_morphisms_open_games}, and reversible
computation~\citep{james2012information}. They are also studied in
K-theory~\citep{guillou2009strictification,gomez2009fibered}. Sheet
diagrams could potentially be used in each of these fields, but one
important obstacle for using these diagrams is the difficulty
of typesetting and manipulating them.  We introduce a web-based tool called
SheetShow,\footnote{Available at \url{https://wetneb.github.io/sheetshow/} (web app) and \url{https://github.com/wetneb/sheetshow} (source code)} which
renders sheet diagrams as vector graphics based on a purely
combinatorial description of their topology.  We give an overview of
the data structures of this tool in Appendix~\ref{app:sheetshow}, which are based on work by~\cite{delpeuch2018normalization-1}.

\section{Acknowledgements}

The authors would like to thank Sam Staton, Jamie Vicary, Lê Thành Dung Nguyên and Spencer Breiner for their insightful comments on this project, Bill Dwyer for his vector 3D graphics renderer \texttt{seenjs} and Jan Vaillant for his Javascript port of the GNU Linear Programming Kit. Antonin Delpeuch is supported by an EPSRC scholarship.

\section{Bimonoidal categories}

\begin{definition} \label{def:bimonoidal}
  A \emph{bimonoidal category}, or \emph{rig category}, is a category $\mathcal{C}$
  with a monoidal structure $(\mathcal{C}, \cdot, I)$ and a symmetric monoidal structure $(\mathcal{C}, \oplus, O)$
  with natural isomorphisms called the \emph{left and right distributors}:
  \begin{align*}
  \delta_{A,B,C} : A (B \oplus C) \rightarrow AB \oplus AC \\
  \delta^{\#}_{A,B,C} : (A \oplus B) C \rightarrow AC \oplus BC
  \end{align*}
  and isomorphisms called the \emph{left and right annihilator}:
  \begin{align*}
    \lambda^*_A : O A \rightarrow O \\
    \rho^*_A : A O \rightarrow O 
  \end{align*}
  satisfying the coherence conditions given in Appendix~\ref{app:coherence-axioms}.
\end{definition}

For instance, $\mathbf{Set}$ is bimonoidal when equipped with the monoidal structures $(\mathbf{Set}, \times, \{ \star \})$ and $(\mathbf{Set}, \sqcup, \emptyset)$, where $\times$ is the cartesian product and $\sqcup$ is the disjoint union.
Similarly, $\mathbf{Vect}_k$ (the category of vector spaces on a field $k$) is bimonoidal for $(\mathbf{Vect}, \otimes, k)$ and $(\mathbf{Vect}, \oplus, O)$.

One important question when defining categorical structures such as
this one is whether they satisfy coherence. By coherence, we mean that
any two parallel morphisms generated by the structural isomorphisms
$\delta, \delta^{\#}, \lambda^*, \rho^*$ and those of the monoidal
structures are equal.  For bimonoidal categories, this unfortunately
does not hold. The symmetry of $(\mathcal{C}, \oplus, O)$ makes
this impossible, as the isomorphisms $\gamma_{A,A}$ and $1_A \oplus 1_A$
are parallel and distinct.

However, a restricted form of coherence for bimonoidal categories was proved by
\citet{laplaza1972coherence}. It applies when the domain (or
equivalently codomain) of the parallel pair of morphisms is
\emph{regular}, which essentially means that no two occurrences of the
same object generator can be swapped by a symmetry. We will make this
precise in Section~\ref{sec:bimonoidal-signatures} as
Theorem~\ref{thm:regular-coherence}.
We finish this section with a semi-strictification theorem for bimonoidal categories.
\begin{definition}
  A bimonoidal category is \emph{left-semistrict} (resp. \emph{right-semistrict}) if both monoidal
  structures are strict and $\delta^\#_{A,B,C}$ (resp. $\delta_{A,B,C}$) is an identity for all $A, B, C$. It is \emph{semistrict} if it is either left- or right-semistrict.
\end{definition}

\begin{theorem}[{\citealp{guillou2009strictification}}] \label{thm:semi-strictification}
  Any bimonoidal category is equivalent to a semistrict one.
\end{theorem}

\section{Symmetric monoidal string diagrams}

\begin{figure}
  \centering

  \begin{subfigure}{0.4\textwidth}
    \centering
    \vspace{0.7cm}
    \def\svgscale{0.7}
    \small
    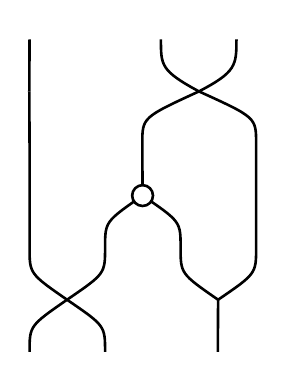
    \vspace{0.8cm}
    \label{fig:example-symmetric-diagram}
  \end{subfigure}
  \begin{subfigure}{0.4\textwidth}
    \centering
    \def\svgscale{0.7}
    \small
\begingroup%
  \makeatletter%
  \providecommand\color[2][]{%
    \errmessage{(Inkscape) Color is used for the text in Inkscape, but the package 'color.sty' is not loaded}%
    \renewcommand\color[2][]{}%
  }%
  \providecommand\transparent[1]{%
    \errmessage{(Inkscape) Transparency is used (non-zero) for the text in Inkscape, but the package 'transparent.sty' is not loaded}%
    \renewcommand\transparent[1]{}%
  }%
  \providecommand\rotatebox[2]{#2}%
  \newcommand*\fsize{\dimexpr\f@size pt\relax}%
  \newcommand*\lineheight[1]{\fontsize{\fsize}{#1\fsize}\selectfont}%
  \ifx\svgwidth\undefined%
    \setlength{\unitlength}{86.38406324bp}%
    \ifx\svgscale\undefined%
      \relax%
    \else%
      \setlength{\unitlength}{\unitlength * \real{\svgscale}}%
    \fi%
  \else%
    \setlength{\unitlength}{\svgwidth}%
  \fi%
  \global\let\svgwidth\undefined%
  \global\let\svgscale\undefined%
  \makeatother%
  \begin{picture}(1,1.98379578)%
    \lineheight{1}%
    \setlength\tabcolsep{0pt}%
    \put(0,0){\includegraphics[width=\unitlength,page=1]{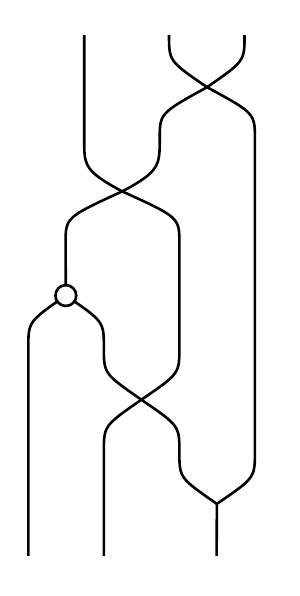}}%
    \put(0.30632618,0.99846356){\makebox(0,0)[lt]{\lineheight{1.25}\smash{\begin{tabular}[t]{l}$f$\end{tabular}}}}%
    \put(0,0){\includegraphics[width=\unitlength,page=2]{isomorphic_symmetric_monoidal.pdf}}%
    \put(0.80989128,0.29088419){\makebox(0,0)[lt]{\lineheight{1.25}\smash{\begin{tabular}[t]{l}$g$\end{tabular}}}}%
    \put(0.72031031,0.01796286){\makebox(0,0)[t]{\lineheight{1.25}\smash{\begin{tabular}[t]{c}$C$\end{tabular}}}}%
    \put(0.34990531,0.01609689){\makebox(0,0)[t]{\lineheight{1.25}\smash{\begin{tabular}[t]{c}$B$\end{tabular}}}}%
    \put(0.09329927,0.01531248){\makebox(0,0)[t]{\lineheight{1.25}\smash{\begin{tabular}[t]{c}$A$\end{tabular}}}}%
    \put(0.28067006,1.89580963){\makebox(0,0)[t]{\lineheight{1.25}\smash{\begin{tabular}[t]{c}$B$\end{tabular}}}}%
    \put(0.56210437,1.90022427){\makebox(0,0)[t]{\lineheight{1.25}\smash{\begin{tabular}[t]{c}$B$\end{tabular}}}}%
    \put(0.81815454,1.90463893){\makebox(0,0)[t]{\lineheight{1.25}\smash{\begin{tabular}[t]{c}$C$\end{tabular}}}}%
  \end{picture}%
\endgroup%

    \label{fig:isomorphic-symmetric-monoidal}
  \end{subfigure}
  \caption{Two isomorphic symmetric monoidal string diagrams}
  \label{fig:example-symmetric-diagrams}
\end{figure}

In this section we briefly recall results on string diagrams for symmetric monoidal categories.  A string diagram in a symmetric monoidal category is given in Figure~\ref{fig:example-symmetric-diagrams}.  The string diagram on the left represents the morphism $(1_B \otimes \gamma_{C,B}) \circ (1_B \otimes f \otimes 1_B) \circ (\gamma_{A,B} \otimes g)$, where $f : A \otimes A \rightarrow C$, $g : C \rightarrow A \otimes B$, where $\gamma_{A,B} : A \otimes B \rightarrow B \otimes A$ is the symmetry.  

\begin{definition}
  A \emph{monoidal signature} $\Sigma$ is given by a set of \emph{object symbols} $\Ob \Sigma$, a set of \emph{morphism symbols} $\Mor \Sigma$ and domain and codomain functions $\dom, \cod : \Mor \Sigma \rightarrow (\Ob \Sigma)^*$, where $(\Ob \Sigma)^*$ is the set of finite words on $\Ob \Sigma $.
\end{definition}

\noindent For instance, we can define a monoidal signature $\Sigma$ as:
\begin{align*}
  \Ob \Sigma &= \{A, B, C\} \\
  \Mor \Sigma  &= \{ f, g \} \\
  \dom(f) &= [A, A ] \\
  \cod(f) &= [ C ] \\
  \dom(g) &= [ C ] \\
  \cod(g) &= [A, B]
\end{align*}

\begin{definition}
  A \emph{symmetric monoidal string diagram} on a symmetric monoidal signature $\Sigma$ is an anchored progressive polarised diagram of $\Sigma$ in the sense of \citet{joyal1991geometry-1}.
\end{definition}

\noindent In the example signature $\Sigma$ above, one can draw the string diagrams of Figure~\ref{fig:example-symmetric-diagrams}.

\begin{definition}
  An isomorphism of symmetric monoidal string diagrams $\phi : D \rightarrow D'$ is an anchored isomorphism of progressive plane diagrams in the sense of \citet{joyal1991geometry-1}.
\end{definition}

The two example diagrams in Figure~\ref{fig:example-symmetric-diagrams} are isomorphic. Given a monoidal signature $\Sigma$, one can form a category $\mathbb{F}_s(\Sigma)$ whose objects are $(\Ob \Sigma)^*$ and morphisms are isomorphism classes of symmetric monoidal string diagrams on $\Sigma$.

\begin{theorem} \citep{joyal1991geometry-1,selinger2010survey-1} \label{thm:joyal-street}
  A well-formed equation between morphisms in the language of symmetric monoidal categories
follows from the axioms of symmetric monoidal categories if and only if it holds, up to
isomorphism of diagrams, in the graphical language. In other words, $\mathbb{F}_s(\Sigma)$ is the free symmetric monoidal category on $\Sigma$.
\end{theorem}

Notice that the objects of the free symmetric monoidal category, $\Ob (\Sigma)^*$, correspond to elements of the free monoid on the object symbols of the signature. (The fact that the set of finite words is the free monoid can be seen as a 1-dimensional analogue of the previous theorem.)

The aim of this article is to provide analogous results for bimonoidal categories: we characterize the free bimonoidal category on a signature as a certain category of diagrams quotiented by certain topological equivalences.

\section{Bimonoidal signatures} \label{sec:bimonoidal-signatures}

Given a set $X$, let $\mathrm{R}(X)$ denote the set of expressions generated by the two binary operators $\cdot$ and $\oplus$, the two nullary symbols $I$ and $O$, and elements of $X$ as nullary symbols.  Call $\mathrm{R}(X)$ the set of bimonoidal object expressions on generators $X$.

\begin{definition} \label{def:bimonoidal-signature}
  A \emph{bimonoidal signature} $\Sigma$ consists of a set $\Ob \Sigma$ of \emph{object symbols}, a set $\Mor \Sigma $ of \emph{morphism symbols}, together with domain and codomain functions $\dom, \cod : \Mor \Sigma \rightarrow R(\Ob \Sigma)$.
  We write $f : \dom(f) \to \cod(f)$.
\end{definition}

One can then give a tautological definition of the free bimonoidal category such a signature generates.

\begin{definition}
  Given a bimonoidal signature $\Sigma$, define the \emph{free bimonoidal category}, $\overline{\Sigma}$, that it generates:  the objects of $\overline{\Sigma}$ are given by $R(\Ob \Sigma)$; the morphisms are equivalence classes of morphism expressions built out of $\Mor \Sigma$, structural isomorphisms from Definition~\ref{def:bimonoidal} and identities quotiented by the axioms of bimonoidal categories.
\end{definition}

This definition is not particularily easy to work with, since morphism expressions can easily get cluttered with structural isomorphisms due to the interplay between the three binary operations involved: $\circ$, $\cdot$ and $\oplus$. By defining sheet diagrams for bimonoidal categories, we will provide a more practical description of $\overline{\Sigma}$, where the bimonoidal axioms are interpreted as topological deformations.

Definition~\ref{def:bimonoidal-signature} is the most general form of a bimonoidal signature, but again not the most convenient to work with. Just like monoidal signatures normalize the domains and codomains of their morphism symbols by forgetting the
bracketing of their products, we will introduce a similar normalization for bimonoidal signatures. First, by Theorem~\ref{thm:semi-strictification} we can assume up to equivalence that both the additive and multiplicative monoidal structures are strict. We
therefore omit the associators and unitors in what follows.

\begin{definition} \label{def:normalized}
  An expression $\phi \in \mathrm{R}(X)$ is \emph{normalized} when it is a sum of products of generators (elements of $X$). Each summand can have no factors (in which case the summand is simply $I$) and the sum can have no summands, in which case $\phi = O$.
  A bimonoidal signature is normalized when all its domains and codomains of morphism symbols are normalized.  
\end{definition}

For any bimonoidal signature, we want to define a normalized version of it. This requires a bit more care than in the monoidal case because there can be multiple ways to normalize an object expression. For instance, the expression $(A \oplus B)(C \oplus D)$ is equivalent to two normalized forms: $AC \oplus AD \oplus BC \oplus BD$ and $AC \oplus BC \oplus AD \oplus BD$.

\begin{definition}
  For any expression $A \in \mathrm{R}(X)$ we define its normal
  form $N(A) = \bigoplus_i A_i$ where $A_i$ are products, by induction:
  \begin{itemize}
  \item $N(O) = O$
  \item $N(I) = I$
  \item $N(A \oplus B) = N(A) \oplus N(B)$
  \item $N(A \cdot B) = \bigoplus_i \bigoplus_j A_i B_j$, where $N(A) = \bigoplus_i A_i$ and $N(B) = \bigoplus_j B_j$ with $A_i$, $B_j$ products of generators.
  \end{itemize}
\end{definition}

\noindent For example, $N ((A \oplus B) (C \oplus D)) = AC \oplus AD \oplus BC \oplus BD$.

\begin{definition}
  Let $A_1, \dots, A_p$ and $B_1, \dots, B_q$ be object expressions. We define an isomorphism $\Delta_{p,q} : (\bigoplus_i A_i) (\bigoplus_j B_j) \rightarrow \bigoplus_i \bigoplus_j A_i B_j$ by repeated application of the distributor $\delta$. Formally, the definition is by induction on $p$ and $q$:
  \begin{itemize}
  \item For $p = 0$, $\Delta_{0,q} = \lambda^*_{\oplus_j B_j} : O (\bigoplus_j B_j) \to O$
  \item For $p = 1$, $\Delta_{p,q}$ is obtained by repeated applications of the distributor $\delta$. More precisely:
  \item For $p = 1$ and $q = 0$, $\Delta_{1,0} = \rho^*_{A_1} : A_1 O \to O$
  \item For $p = 1$ and $q > 0$, \begin{align*}
  	\Delta_{1, q} : A_1 (\textstyle \bigoplus_{j = 1}^q B_j) =\ &A_1 (\textstyle \bigoplus_{j = 1}^{q - 1} B_j \oplus B_q) \\
	\xrightarrow{\delta_{A_1, \bigoplus_{j = 1}^{q - 1} B_j, B_q}}\ &\textstyle A_1 \bigoplus_{j = 1}^{q - 1} B_j \oplus A_1 B_q \\
	\xrightarrow{\Delta_{1, q - 1} \oplus 1_{A_1 B_q}}\ &\textstyle \bigoplus_{j = 1}^{q - 1} A_1 B_j \oplus A_1 B_q \\
	=\ &\textstyle \bigoplus_{j = 1}^q A_1 B_j
  \end{align*}

  \item For $p > 1$, $(\bigoplus_i A_i)(\bigoplus_j B_j) \rightarrow_{\delta} (\bigoplus_{i=1}^{p-1} A  _i)(\bigoplus_j B_j) \oplus A_p (\bigoplus_j B_j)$. We can apply $\Delta_{p-1,q}$ on the left-hand side and an iteration of $\delta$ on the right-hand side, by a similar induction on $q$. 
  \end{itemize}
\end{definition}

\begin{definition}
  For any object expression $A \in \mathrm{R}(X)$ we define its normalization morphism $n_A : A \rightarrow N(A)$ by induction:
  \begin{itemize}
  \item $n_O = 1_O$
  \item $n_I = 1_I$
  \item $n_{A \oplus B} = n_A \oplus n_B$
  \item $n_{A \cdot B} = \Delta_{p,q} \circ (n_A \cdot n_B)$ where $p$ and $q$ are the number of summands in $N(A)$ and $N(B)$ respectively.
  \end{itemize}
\end{definition}

\noindent Note that every normalization morphism is an isomorphism.

\begin{definition}
  Given any bimonoidal signature $\Sigma$, we can build a normalized signature $N(\Sigma)$
  where all domains and codomains are normalized by $N$:
  \begin{align*}
    \Ob N(\Sigma) &= \Ob \Sigma \\
    \Mor N(\Sigma) &= \{ f' : N(\dom f) \rightarrow N(\cod f) \mid f \in \Mor \Sigma \}
  \end{align*}
\end{definition}

\begin{theorem}
  For any bimonoidal signature $\Sigma$, the categories $\overline{\Sigma}$ and $\overline{N(\Sigma)}$
  are bimonoidally isomorphic.
\end{theorem}

\begin{proof}
  We define an identity on objects functor $U:\overline \Sigma\to \overline{N(\Sigma)}$, which translates
  a morphism in $\overline{\Sigma}$ to $\overline{N(\Sigma)}$, by structural
  induction on the morphism expression.
  For the base case, for any $\phi$ in $\Mor \Sigma$, define $U(\phi):=n^{-1}_{\cod \phi} \circ \phi' \circ n_{\dom \phi}$.
  Moreover, for any object $A$, define $U(1_A):=1_A$.  And similarly, for all of the structural natural isomorphisms of a
  bimonoidal category, define $U(\delta_{A,B,C}) := \delta_{A,B,C}$, $U(\delta^\#_{A,B,C}) := \delta^\#_{A,B,C}$ and so on.
  For the inductive case, consider some morphisms
  $\phi_1:A\to B$, $\phi_2:B\to C$ in $\overline \Sigma$ for which $U(\phi_1)$ and $U(\phi_2)$ are already defined.  Then define
  $U(\phi_2 \circ \phi_1) := U(\phi_2) \circ U(\phi_1)$.  Similarly, consider  some morphisms
  $\phi_1:A\to B$, $\phi_2:C\to D$ in $\overline \Sigma$ for which $U(\phi_1)$ and $U(\phi_2)$ are already defined. 
  Then define $U(\phi_1 \oplus \phi_2):=U(\phi_1)\oplus U(\phi_2)$ and  $U(\phi_1 \cdot \phi_2):=U(\phi_1)\cdot U(\phi_2)$.

  Similarly, define an identity on objects functor $V: \overline{N(\Sigma)}\to \overline \Sigma$ in a completely analagous way, by structural induction on the signature of $\Mor N(\Sigma)$ and the axioms of a bimonoidal category.  For the base case, this functor takes generators $V(\phi'):=n_{\cod \phi} \circ \phi \circ n_{\dom \phi}^{-1}$.  The rest of the induction is defined as before.

  The only thing left to prove is that the two assignments between $\Sigma$ and  $\overline{N(\Sigma)}$ are well defined; as they preserve the bimonoidal structure and are inverse to each other by construction.  That is to say, we have to show that they preserve the equivalence relations on morphisms induced by the axioms of a bimonoidal category.
Because both assignments only replace generators by expressions, the equivalence relation induced by the bimonoidal structure is automatically preserved. The two functors are bimonoidal (in the sense of Appendix~\ref{app:bimonoidal-functor}) and mutually inverse.
\end{proof}

As a consequence, we will only consider normalized signatures in what
follows, as this will not restrict the generality of our results.

Let us turn to coherence. As mentioned earlier, coherence does not
hold for bimonoidal categories in general, but only when some conditions
on their domain (or equivalently codomain) are met.

\begin{definition}[{\citealp{laplaza1972coherence}}] \label{defi:regular}
  An object $A \in \Ob \Sigma$ is regular when all the summands in $N(A)$
  are distinct (they are all different lists of generators) and for each summand of $N(A)$, its factors are all different (it is a product of distinct generators).
\end{definition}

For instance, $A B \oplus B A$ is regular, but $A B \oplus A B$ and $A
A \oplus A B$ are not.  Note that the second condition (each summand
being a product of distinct generators) was required by Laplaza
because they assumed the multiplicative monoidal structure to be
symmetric. We only assume the additive structure to be symmetric, so
this condition should be superfluous in our case.  We keep it for the
sake of accurate citation as we will be able to accommodate this
artificial restriction later on, when using the following theorem:

\begin{theorem}[{\citealp{laplaza1972coherence}}] \label{thm:regular-coherence}
  Let $A$, $B$ be regular objects of $\overline{\Sigma}$. For all morphisms
  $f, g : A \rightarrow B$ generated by structural isomorphisms of $\overline{\Sigma}$,
  $f = g$.
\end{theorem}

We finish this section with a lemma on the normalization function $N$.

\begin{lemma} \label{lemma:associativity-n}
  For all objects $A, B, C \in \Ob \overline{\Sigma}$,
  $$N(A \cdot N(B \cdot C)) = N(N(A \cdot B) \cdot C)$$  
\end{lemma}

\begin{proof}
  Let $N(A) = \bigoplus_i A_i$, $N(B) = \bigoplus_j B_j$, $N(C) = \bigoplus_k C_k$. Then  
  \begin{align*}
    N(A \cdot N(B \cdot C)) &= N(A \cdot (B \cdot C)) \\
    &= \bigoplus_i \bigoplus_{j,k} A_i (B_j C_k) \\
    &= \bigoplus_{i,j} \bigoplus_k (A_i B_j) C_k \\
    &= N((A \cdot B) \cdot C) \\
    &= N(N(A \cdot B) \cdot C) \qedhere
  \end{align*}
\end{proof}

\section{Sheet diagrams}

In this section, we assume a fixed normalized bimonoidal signature $\Sigma$.
Because bimonoidal categories are symmetric monoidal categories for
their additive structure, we can already use string diagrams for
symmetric monoidal categories to reason about bimonoidal
categories. Such a diagrammatic language treats the multiplicative
structure as opaque, but we will see in this section how to extend the
language to take the second monoidal structure into account as well.

The first step consists in defining a monoidal signature which we will
use for our symmetric monoidal diagrams.
\begin{definition} \label{def:gamma}
  The monoidal signature $\Gamma$ is given by
  \begin{align*}
    \Ob \Gamma &= (\Ob \Sigma)^* \\
    \Mor \Gamma &= (\Mor \Sigma \cup \{ 1_A | A \in \Ob \Sigma \})^* \\
    \dom [f_1,\ldots, f_n] &= N[\dom (f_1) \cdot \ldots \cdot \dom (f_n)] \\
    \cod [f_1,\ldots, f_n]  &= N[\cod (f_1) \cdot \ldots \cdot \cod (f_n)]
  \end{align*}
  where $\dom (1_A) = \cod (1_A) = A$.
\end{definition}
We are slightly abusing notation in the previous definition: The normalized domains and codomains are sums of products, which we interpret as lists of lists, that is, as lists of object symbols of $\Gamma$.

\noindent The signature $\Gamma$ generates a symmetric monoidal
category $\mathcal{C}$ for the additive structure. Object generators are
multiplicative products of bimonoidal generators, therefore objects in
$\mathcal{C}$ are sums of products of bimonoidal generators.

Similarly, morphism generators in $\Gamma$ are multiplicative products
of morphism generators in $\Sigma$, except that we also allow for
identities to be part of the product. The domain of a morphism generator
is obtained by normalizing the product of the domains of its components.

For instance, consider object generators $A, B, C, D$ and bimonoidal
morphism generators $f : A \oplus AB \rightarrow C$ and $g : BD
\rightarrow A \oplus D$.  We can form $f \cdot 1_C \cdot g \in \Mor
\Gamma$. Its domain is $N((A \oplus AB)\cdot C \cdot BD) = ACBD \oplus
ABCBD$ and its codomain is $N(C \cdot C \cdot (A \oplus D)) = CCA
\oplus CCD$. When drawn as a string diagram generator, it looks as follows:

\begin{center}
\begingroup%
  \makeatletter%
  \providecommand\color[2][]{%
    \errmessage{(Inkscape) Color is used for the text in Inkscape, but the package 'color.sty' is not loaded}%
    \renewcommand\color[2][]{}%
  }%
  \providecommand\transparent[1]{%
    \errmessage{(Inkscape) Transparency is used (non-zero) for the text in Inkscape, but the package 'transparent.sty' is not loaded}%
    \renewcommand\transparent[1]{}%
  }%
  \providecommand\rotatebox[2]{#2}%
  \newcommand*\fsize{\dimexpr\f@size pt\relax}%
  \newcommand*\lineheight[1]{\fontsize{\fsize}{#1\fsize}\selectfont}%
  \ifx\svgwidth\undefined%
    \setlength{\unitlength}{142.30258083bp}%
    \ifx\svgscale\undefined%
      \relax%
    \else%
      \setlength{\unitlength}{\unitlength * \real{\svgscale}}%
    \fi%
  \else%
    \setlength{\unitlength}{\svgwidth}%
  \fi%
  \global\let\svgwidth\undefined%
  \global\let\svgscale\undefined%
  \makeatother%
  \begin{picture}(1,0.37692107)%
    \lineheight{1}%
    \setlength\tabcolsep{0pt}%
    \put(0,0){\includegraphics[width=\unitlength,page=1]{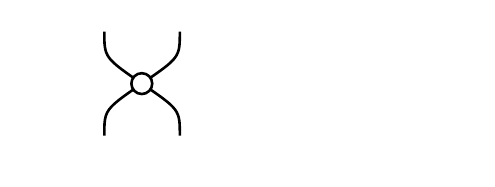}}%
    \put(0.34602038,0.20655148){\makebox(0,0)[lt]{\lineheight{1.25}\smash{\begin{tabular}[t]{l}$f \cdot 1_C \cdot g$\end{tabular}}}}%
    \put(-0.00524987,0.0303772){\makebox(0,0)[lt]{\lineheight{1.25}\smash{\begin{tabular}[t]{l}$ACBD$\end{tabular}}}}%
    \put(0.31099135,0.03125164){\makebox(0,0)[lt]{\lineheight{1.25}\smash{\begin{tabular}[t]{l}$ABCBD$\end{tabular}}}}%
    \put(0.32906348,0.34995114){\makebox(0,0)[lt]{\lineheight{1.25}\smash{\begin{tabular}[t]{l}$CCD$\end{tabular}}}}%
    \put(0.08566938,0.34905335){\makebox(0,0)[lt]{\lineheight{1.25}\smash{\begin{tabular}[t]{l}$CCA$\end{tabular}}}}%
  \end{picture}%
\endgroup%

\end{center}

This is not very informative, since only the additive monoidal
structure is reflected by the diagram. The multiplicative structure is
left unanalyzed because it is internal to the object and morphism
generators.

To fix this issue, we extrude our monoidal diagram into a sheet diagram.
Edges of our monoidal diagram become sheets, and vertices become seams.
On the sheets, we can draw wires whose connectivity reflects which factor
of the morphism generator they are coming from:

\begin{center}
  \scalebox{0.8}{
  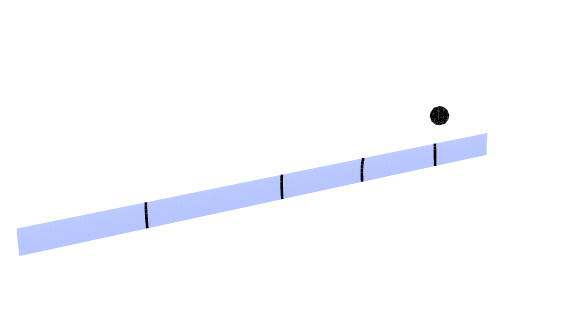}
\end{center}

Informally, each factor of the morphism generator corresponds to a
node on the seam, in the same order. These nodes are represented here
by small black spheres, except for the middle one which corresponds to
an identity, which we leave unmarked. This idea can be
generalized to arbitrary diagrams. This section defines these diagrams
and the associated class of topological transformations.

\begin{definition}
  Let $[f_1,\ldots, f_n] \in \Mor \Gamma$ and let $\dom [f_1,\ldots, f_n]
	= \bigoplus_{j=1}^p \bigotimes_{k=1}^{q_j} A_{jk}$ be its domain. For each $j, k$
  we define the \emph{origin} of $A_{jk}$ in $[f_1,\ldots, f_n]$ as the
  index $1 \leq i \leq n$ such that $A_{jk}$ occurs in $\dom f_i$. If
  $A_{jk}$ occurs in the domain of more than one morphism factor, this
  can be made unambiguous by adding indices to the object symbols
  involved in the domains before normalization.  Similarly, the origin
  of an object symbol in the codomain of a morphism generator is defined.
\end{definition}

\begin{definition}
  Given a symmetric monoidal string diagram $S$ on $\Gamma$, a \emph{sheet diagram}
  $D$ for $S$ is a collection of topological manifolds with boundaries in $\mathbb{R}^3$:
  \begin{itemize}
  \item for each vertex $v \in \mathbb{R}^2$ of $S$, there is a \emph{seam} $v \times [0, 1] \in \mathbb{R}^2$. Let $[f_1,\ldots, f_n]$ be the morphism generators associated with $v$. For each $i$ we pick an $x_{vi} \in (0, 1)$ such that $x_{v1} < \dots < x_{vn}$ and add a \emph{node} $(v, x_{vi})$, which is included in the vertex's seam;
  \item for each edge $e \subset \mathbb{R}^2$ of $S$, there is a \emph{sheet} $e \times [0, 1] \subset \mathbb{R}^3$. Let $p_e : [0, 1] \rightarrow \mathbb{R}^2$ be a parametrization of $e$ from source to target (bottom up).
    Let $[ A_{e1}, \ldots, A_{en}]$ be the type associated to $e$ in $S$. For each $i$ we pick a parametrized segment $\gamma_{ei} : [0,1] \rightarrow [0,1]$ which gives rise to a \emph{wire} $w_{ei} : t \in [0,1] \rightarrow p_e(t) \times \gamma_{ei}(t)$ included in the sheet for edge $e$.
    We require that for all $t \in (0, 1)$ and $i < j$, $\gamma_{ei}(t) < \gamma_{ej}(t)$.
  \end{itemize}
  Furthermore, we require the following conditions:
  \begin{enumerate}
  \item In $S$, if an edge $e$ connects to a vertex $v$ from below (into its domain), then for all wires $w_{ei}$ on the sheet corresponding to $e$ in $D$, $\gamma_{ei}(1) = x_{vj}$ where $j$ is the origin of $A_{ei}$ in the domain of $v$;
  \item In $S$, if an edge $e$ connects to a vertex $v$ from above (out of its codomain), then for all wires $w_{ei}$ on the sheet corresponding to $e$ in $D$, $\gamma_{ei}(0) = x_{vj}$ where $j$ is the origin of $A_{ei}$ in the codomain of $v$.
  \end{enumerate}
  The \emph{skeleton} of $D$ is $S$. The set of sheet diagrams on $\Gamma$ is
  denoted by $D(\Gamma)$.
\end{definition}

\noindent Let us consider an example of a sheet diagram.
\begin{figure}[H]
  \centering
  \begin{subfigure}{0.45\textwidth}
    \centering
\begingroup%
  \makeatletter%
  \providecommand\color[2][]{%
    \errmessage{(Inkscape) Color is used for the text in Inkscape, but the package 'color.sty' is not loaded}%
    \renewcommand\color[2][]{}%
  }%
  \providecommand\transparent[1]{%
    \errmessage{(Inkscape) Transparency is used (non-zero) for the text in Inkscape, but the package 'transparent.sty' is not loaded}%
    \renewcommand\transparent[1]{}%
  }%
  \providecommand\rotatebox[2]{#2}%
  \newcommand*\fsize{\dimexpr\f@size pt\relax}%
  \newcommand*\lineheight[1]{\fontsize{\fsize}{#1\fsize}\selectfont}%
  \ifx\svgwidth\undefined%
    \setlength{\unitlength}{83.55726242bp}%
    \ifx\svgscale\undefined%
      \relax%
    \else%
      \setlength{\unitlength}{\unitlength * \real{\svgscale}}%
    \fi%
  \else%
    \setlength{\unitlength}{\svgwidth}%
  \fi%
  \global\let\svgwidth\undefined%
  \global\let\svgscale\undefined%
  \makeatother%
  \begin{picture}(1,0.9969896)%
    \lineheight{1}%
    \setlength\tabcolsep{0pt}%
    \put(0,0){\includegraphics[width=\unitlength,page=1]{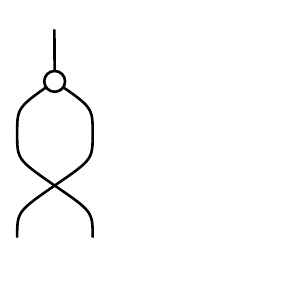}}%
    \put(0.27816039,0.70141465){\makebox(0,0)[lt]{\lineheight{1.25}\smash{\begin{tabular}[t]{l}$g \cdot 1_F$\end{tabular}}}}%
    \put(0,0){\includegraphics[width=\unitlength,page=2]{example-skeleton.pdf}}%
    \put(0.27816039,0.34237945){\makebox(0,0)[lt]{\lineheight{1.25}\smash{\begin{tabular}[t]{l}$f$\end{tabular}}}}%
    \put(0.18831948,0.0517341){\makebox(0,0)[lt]{\lineheight{1.25}\smash{\begin{tabular}[t]{l}$BCD$\end{tabular}}}}%
    \put(-0.00894082,0.05173403){\makebox(0,0)[lt]{\lineheight{1.25}\smash{\begin{tabular}[t]{l}$A$\end{tabular}}}}%
    \put(0.08007511,0.95105834){\makebox(0,0)[lt]{\lineheight{1.25}\smash{\begin{tabular}[t]{l}$HF$\end{tabular}}}}%
  \end{picture}%
\endgroup%

    \vspace{.5cm}  
    \caption{A string diagram $S$ on $\Gamma$}
  \end{subfigure}
  \begin{subfigure}{0.45\textwidth}
    \centering
  \scalebox{0.8}{
    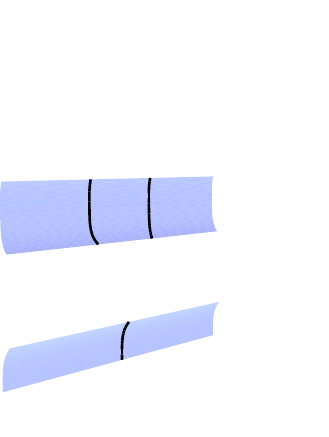}
  \caption{A sheet diagram based on $S$}
  \end{subfigure}
\end{figure}
Obtaining a sheet diagram from a string diagram only requires picking
node positions on each seam, and trajectories of the wires on each
sheet, such that the two boundary conditions are satisfied. The
geometry of the seams and sheets is directly inherited from that of
the string diagram itself. The boundary conditions ensure proper typing
of the diagram. Geometrically, there is a projection from a sheet diagram to its skeleton.

Note that nothing interesting happens on sheets themselves: wires flow
vertically, without being able to cross, from the bottom seam to the
top seam of the sheet (or the diagram boundaries). However, for seams
which have only one input sheet and one output sheet, we will use the
convention of not drawing the seam between the two sheets, informally
treating them as the same sheet. This makes it possible to draw monoidal
string diagrams for the multiplicative product ``on the sheets''.
In the following example we mark the seams with dashed red lines, which will
be omitted in the future:
\begin{figure}[H]
  \centering
  \begin{subfigure}{0.45\textwidth}
    \centering
    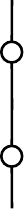
    \vspace{.3cm}
    \caption{Skeleton}
  \end{subfigure}
  \begin{subfigure}{0.45\textwidth}
    \centering
    \scalebox{0.7}{
      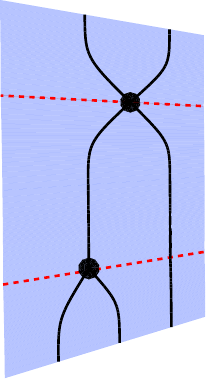
    }
    \caption{Sheet diagram}
  \end{subfigure}
\end{figure}

\begin{definition} \label{def:sheet-diagram-domain}
  Let $D$ be a sheet diagram. Its \emph{domain} $\dom D$ is the domain of its
  skeleton, and similarly for its \emph{codomain} $\cod D$.
\end{definition}

\begin{definition} \label{def:sheet-diagram-composition}
  Let $D_1$ and $D_2$ be sheet diagrams such that $\cod D_1 = \dom D_2$.
  The sheet diagram $D_2 \circ D_1$ is constructed by stacking $D_2$ on top
  of $D_1$, binding sheets and wires at the boundary with linear interpolation.
\end{definition}

\begin{figure}[H]
  \centering
  \begin{tikzpicture}[every node/.style={scale=0.8}]
    \node at (-2,0) {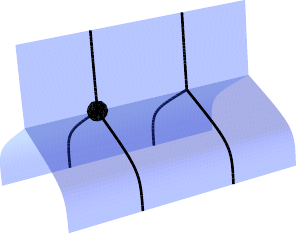};
    \node at (-.25,0) {\Large $\circ$};
    \node at (2,0) {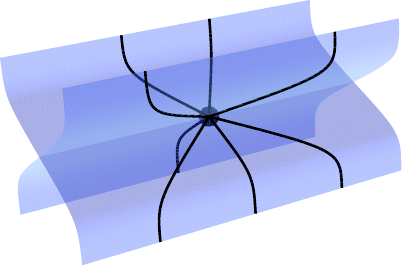};
    \node at (4,0) {\Large $=$};
    \node at (6,0) {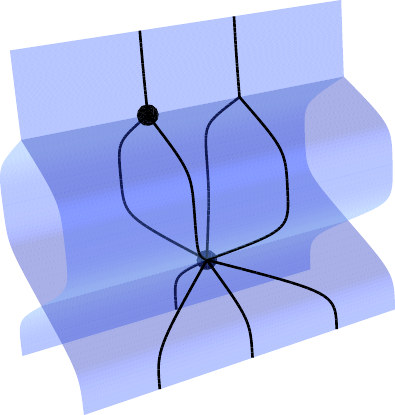};
  \end{tikzpicture}

\caption{Example of composing sheet diagrams with $\circ$}
\end{figure}

\begin{definition} \label{def:sheet-diagram-sum}
  Let $D_1$ and $D_2$ be sheet diagrams. The sheet diagram $D_1 \oplus D_2$ is constructed
  by adjoining $D_2$ to the right of $D_1$. 
\end{definition}

Note that we have $S (D_2 \circ D_1) = S (D_2) \circ S (D_1)$ and $S(D_1 \oplus D_2) = S(D_1) \oplus S(D_2)$.

\begin{figure}[H]
  \centering
  \begin{tikzpicture}[every node/.style={scale=0.8}]
    \node at (-2,0) {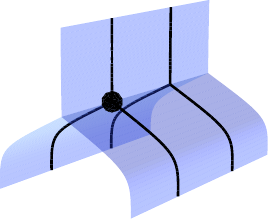};
    \node at (-.25,0) {\Large $\oplus$};
    \node at (1.5,0) {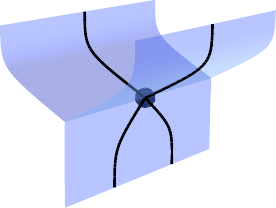};
    \node at (3.25,0) {\Large $=$};
    \node at (6,0) {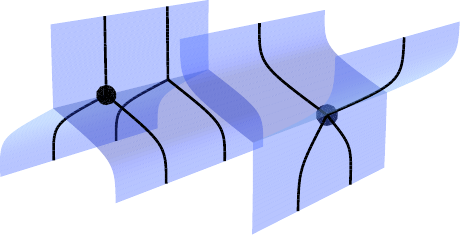};
  \end{tikzpicture}

\caption{Example of composing sheet diagrams with $\oplus$}
\end{figure}

To obtain a bimonoidal category of sheet diagrams, the last binary
operation we need to define is the tensor product. This is more
intricate since it is not naturally represented by the structure of
the skeleton. We start by defining the whiskering of a diagram with an
object, in other words a tensor product where one of the factors is an identity.

\begin{definition} \label{def:sheet-diagram-whiskering}
  Let $A \in \Ob \Sigma$ and $D$ be a sheet diagram. The \emph{left whiskering} of $D$ by $A$, denoted by $A \cdot D$, is obtained from $D$ by:
  \begin{itemize}
  \item adding a node $n_v$ on each seam corresponding to vertex $v$ in $S(D)$, before all other nodes on the seam;
  \item adding a wire $w_e$ on each sheet corresponding to edge $e$ in $S(D)$, before all other wires on the seam;
  \item replacing all symmetries $\gamma_{U,V}$ in $S(D)$ by the whiskered symmetry $\gamma_{AU,AV}$;
  \end{itemize}
  Such that if sheet $e$ connects to seam $v$, wire $w_e$ connects to
  node $n_v$.  We have $A \cdot D : N(A \cdot \dom D) \rightarrow N(A
  \cdot \cod D)$.  The \emph{right whiskering} is defined similarly,
  placing the new nodes and wires after the existing ones instead.
\end{definition}

\begin{figure}[H]
  \centering
  \begin{tikzpicture}[every node/.style={scale=0.8}]
    \node at (-1,0) {E};
    \node at (-.25,0) {\Large $\cdot$};
    \node at (1.5,0) {
\begingroup%
  \makeatletter%
  \providecommand\color[2][]{%
    \errmessage{(Inkscape) Color is used for the text in Inkscape, but the package 'color.sty' is not loaded}%
    \renewcommand\color[2][]{}%
  }%
  \providecommand\transparent[1]{%
    \errmessage{(Inkscape) Transparency is used (non-zero) for the text in Inkscape, but the package 'transparent.sty' is not loaded}%
    \renewcommand\transparent[1]{}%
  }%
  \providecommand\rotatebox[2]{#2}%
  \newcommand*\fsize{\dimexpr\f@size pt\relax}%
  \newcommand*\lineheight[1]{\fontsize{\fsize}{#1\fsize}\selectfont}%
  \ifx\svgwidth\undefined%
    \setlength{\unitlength}{87.60428694bp}%
    \ifx\svgscale\undefined%
      \relax%
    \else%
      \setlength{\unitlength}{\unitlength * \real{\svgscale}}%
    \fi%
  \else%
    \setlength{\unitlength}{\svgwidth}%
  \fi%
  \global\let\svgwidth\undefined%
  \global\let\svgscale\undefined%
  \makeatother%
  \begin{picture}(1,0.82162645)%
    \lineheight{1}%
    \setlength\tabcolsep{0pt}%
    \put(0,0){\includegraphics[width=\unitlength,page=1]{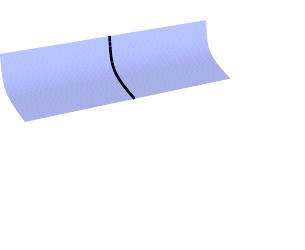}}%
    \put(0.66028543,0.08889778){\color[rgb]{0,0,0}\makebox(0,0)[t]{\lineheight{1.25}\smash{\begin{tabular}[t]{c}$B$\end{tabular}}}}%
    \put(0,0){\includegraphics[width=\unitlength,page=2]{whiskering-1.pdf}}%
    \put(0.35565238,0.75089602){\color[rgb]{0,0,0}\makebox(0,0)[t]{\lineheight{1.25}\smash{\begin{tabular}[t]{c}$C$\end{tabular}}}}%
    \put(0,0){\includegraphics[width=\unitlength,page=3]{whiskering-1.pdf}}%
    \put(0.39712526,0.00869499){\color[rgb]{0,0,0}\makebox(0,0)[t]{\lineheight{1.25}\smash{\begin{tabular}[t]{c}$A$\end{tabular}}}}%
    \put(0,0){\includegraphics[width=\unitlength,page=4]{whiskering-1.pdf}}%
    \put(0.66641465,0.70390129){\color[rgb]{0,0,0}\makebox(0,0)[t]{\lineheight{1.25}\smash{\begin{tabular}[t]{c}$D$\end{tabular}}}}%
  \end{picture}%
\endgroup%
};
    \node at (3.5,0) {\Large $=$};
    \node at (6,0) {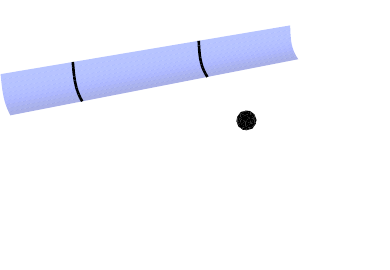};
  \end{tikzpicture}

\caption{Example of the left whiskering of a sheet diagram}
\end{figure}

We now turn to the general definition of the tensor product. We first
need to define a family of isomorphisms to reorder summands.

\begin{definition} \label{def:reorder}
  Let $p, q \in \mathbb{N}$ and $X_1, \dots, X_p, Y_1, \dots, Y_q \in \Ob \Gamma$.
  We define an isomorphism
  $$E_{XY} : \bigoplus_{i=1}^p \bigoplus_{j=1}^q X_i Y_j \rightarrow \bigoplus_{j=1}^q \bigoplus_{i=1}^p X_i Y_j$$
  which reorders the summands as indicated by the commutation of sums,
  by repeated application of the symmetry for $\oplus$.
\end{definition}

The definition of the tensor product of arbitrary diagrams exploits the exchange law
to express it as a composition of whiskered diagrams.

\begin{definition} \label{def:sheet-diagram-tensor}
  Let $f : \bigoplus_i A_i \rightarrow \bigoplus_j B_j$ and $g : \bigoplus_k C_k \rightarrow \bigoplus_l D_l$ be sheet diagrams, where $A_i, B_j, C_k, D_l \in \Ob \Gamma$.
  We can define the tensor product of $D_1$ and $D_2$ as:

  \begin{align*}
    f \otimes g &\coloneqq E_{BD}^{-1} \circ (\bigoplus_l f D_l) \circ E_{AD} \circ (\bigoplus_i A_i g) \\
    &: \bigoplus_i \bigoplus_k A_i C_k \rightarrow \bigoplus_l \bigoplus_j B_j D_l
  \end{align*}
\end{definition}

\begin{figure}[H]
  \centering
  \begin{tikzpicture}[every node/.style={scale=0.8}]
    \node at (-2,0) {
\begingroup%
  \makeatletter%
  \providecommand\color[2][]{%
    \errmessage{(Inkscape) Color is used for the text in Inkscape, but the package 'color.sty' is not loaded}%
    \renewcommand\color[2][]{}%
  }%
  \providecommand\transparent[1]{%
    \errmessage{(Inkscape) Transparency is used (non-zero) for the text in Inkscape, but the package 'transparent.sty' is not loaded}%
    \renewcommand\transparent[1]{}%
  }%
  \providecommand\rotatebox[2]{#2}%
  \newcommand*\fsize{\dimexpr\f@size pt\relax}%
  \newcommand*\lineheight[1]{\fontsize{\fsize}{#1\fsize}\selectfont}%
  \ifx\svgwidth\undefined%
    \setlength{\unitlength}{61.38458282bp}%
    \ifx\svgscale\undefined%
      \relax%
    \else%
      \setlength{\unitlength}{\unitlength * \real{\svgscale}}%
    \fi%
  \else%
    \setlength{\unitlength}{\svgwidth}%
  \fi%
  \global\let\svgwidth\undefined%
  \global\let\svgscale\undefined%
  \makeatother%
  \begin{picture}(1,1.1554)%
    \lineheight{1}%
    \setlength\tabcolsep{0pt}%
    \put(0,0){\includegraphics[width=\unitlength,page=1]{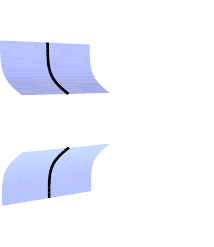}}%
    \put(0.23200689,0.06123963){\color[rgb]{0,0,0}\makebox(0,0)[t]{\lineheight{1.25}\smash{\begin{tabular}[t]{c}$A_1$\end{tabular}}}}%
    \put(0,0){\includegraphics[width=\unitlength,page=2]{tensor-1.pdf}}%
    \put(0.22173968,1.02452954){\color[rgb]{0,0,0}\makebox(0,0)[t]{\lineheight{1.25}\smash{\begin{tabular}[t]{c}$B_1$\end{tabular}}}}%
    \put(0,0){\includegraphics[width=\unitlength,page=3]{tensor-1.pdf}}%
    \put(0.5096063,0.73199216){\color[rgb]{0,0,0}\makebox(0,0)[t]{\lineheight{1.25}\smash{\begin{tabular}[t]{c}$f$\end{tabular}}}}%
    \put(0,0){\includegraphics[width=\unitlength,page=4]{tensor-1.pdf}}%
    \put(0.8262609,0.01240896){\color[rgb]{0,0,0}\makebox(0,0)[t]{\lineheight{1.25}\smash{\begin{tabular}[t]{c}$A_2$\end{tabular}}}}%
    \put(0.82003472,1.05445789){\color[rgb]{0,0,0}\makebox(0,0)[t]{\lineheight{1.25}\smash{\begin{tabular}[t]{c}$B_2$\end{tabular}}}}%
  \end{picture}%
\endgroup%
};
    \node at (-.5,0) {\Large $\otimes$};
    \node at (1,0) {
\begingroup%
  \makeatletter%
  \providecommand\color[2][]{%
    \errmessage{(Inkscape) Color is used for the text in Inkscape, but the package 'color.sty' is not loaded}%
    \renewcommand\color[2][]{}%
  }%
  \providecommand\transparent[1]{%
    \errmessage{(Inkscape) Transparency is used (non-zero) for the text in Inkscape, but the package 'transparent.sty' is not loaded}%
    \renewcommand\transparent[1]{}%
  }%
  \providecommand\rotatebox[2]{#2}%
  \newcommand*\fsize{\dimexpr\f@size pt\relax}%
  \newcommand*\lineheight[1]{\fontsize{\fsize}{#1\fsize}\selectfont}%
  \ifx\svgwidth\undefined%
    \setlength{\unitlength}{61.38458282bp}%
    \ifx\svgscale\undefined%
      \relax%
    \else%
      \setlength{\unitlength}{\unitlength * \real{\svgscale}}%
    \fi%
  \else%
    \setlength{\unitlength}{\svgwidth}%
  \fi%
  \global\let\svgwidth\undefined%
  \global\let\svgscale\undefined%
  \makeatother%
  \begin{picture}(1,1.1554)%
    \lineheight{1}%
    \setlength\tabcolsep{0pt}%
    \put(0,0){\includegraphics[width=\unitlength,page=1]{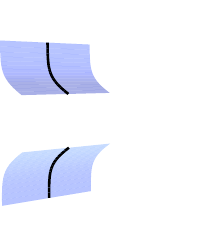}}%
    \put(0.23200689,0.06123963){\color[rgb]{0,0,0}\makebox(0,0)[t]{\lineheight{1.25}\smash{\begin{tabular}[t]{c}$C_1$\end{tabular}}}}%
    \put(0,0){\includegraphics[width=\unitlength,page=2]{tensor-2.pdf}}%
    \put(0.22173968,1.02452954){\color[rgb]{0,0,0}\makebox(0,0)[t]{\lineheight{1.25}\smash{\begin{tabular}[t]{c}$D_1$\end{tabular}}}}%
    \put(0,0){\includegraphics[width=\unitlength,page=3]{tensor-2.pdf}}%
    \put(0.5096063,0.73199216){\color[rgb]{0,0,0}\makebox(0,0)[t]{\lineheight{1.25}\smash{\begin{tabular}[t]{c}$g$\end{tabular}}}}%
    \put(0,0){\includegraphics[width=\unitlength,page=4]{tensor-2.pdf}}%
    \put(0.8262609,0.01240896){\color[rgb]{0,0,0}\makebox(0,0)[t]{\lineheight{1.25}\smash{\begin{tabular}[t]{c}$C_2$\end{tabular}}}}%
    \put(0.82003472,1.05445789){\color[rgb]{0,0,0}\makebox(0,0)[t]{\lineheight{1.25}\smash{\begin{tabular}[t]{c}$D_2$\end{tabular}}}}%
  \end{picture}%
\endgroup%
};
    \node at (2.5,0) {\Large $=$};
    \node at (5.5,0) {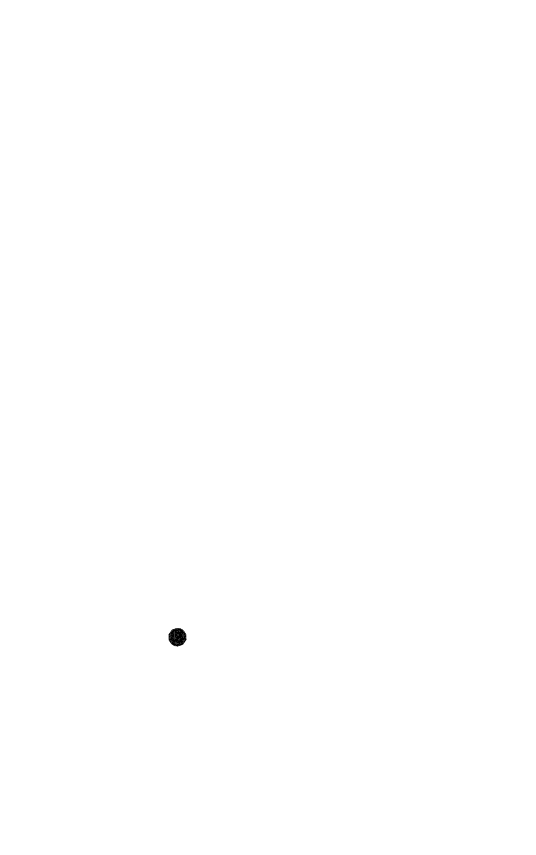};
  \end{tikzpicture}

\caption{Example of the tensoring sheet diagrams}
\end{figure}

Note that we made a choice here: $f \otimes g$ could have equivalently
been defined as $(\bigoplus_j B_j g) \circ E^{-1}_{BC} \circ
(\bigoplus_k f C_k) \circ E^{-1}_{AC}$.  In the next section, we will
define a class of isotopies which will let us relate the two
expressions (Lemma~\ref{lemma:tensor-recursion}).  Before that, we
introduce one last product, that of morphism generators, for which we
do not need to resort to whiskering.

\begin{definition}
	Let $f : \bigoplus_k A_k \rightarrow \bigoplus_l B_l$ and $g : \bigoplus_p C_p \rightarrow \bigoplus_q D_q \in \Mor \Gamma$. They are both lists of morphism generators in $\Sigma$ or identities: $[ f_1, \ldots,  f_n ]$ and $[g_1, \ldots, g_m]$. We define their tensor as
	$$[f_1, \ldots, f_n, g_1, \ldots, g_m]$$
  In other words we concatenate the lists of generators and identities that compose them.
\end{definition}

\begin{lemma}
  For all $f, g \in \Mor \Gamma$, $\dom fg = N((\dom f)(\dom g))$ and $\cod fg = N((\cod f)(\cod g))$.
\end{lemma}

\begin{proof}
  This is a direct consequence of the associativity of N (Lemma~\ref{lemma:associativity-n}) and the definition of domains and codomaing in $\Gamma$ (Definition~\ref{def:gamma}).
\end{proof}

\section{Isomorphisms of sheet diagrams}

We first begin with a lemma on isomorphisms of monoidal string diagrams.
\begin{lemma} \label{lemma:open-graph-iso}
  Given an anchored isomorphism of progressive plane diagrams between
  diagrams $S_1$ and $S_2$, there is an open graph isomorphism $\phi$
  between the open graphs defined by $S_1$ and $S_2$.  In other words
  there are bijections between their sets of nodes, their sets of
  edges, and those respect the adjacency relations.
\end{lemma}

\begin{definition}
	Given sheet diagrams $D_1, D_2$ with $\dom D_1 = \dom D_2$ and $\cod D_1 = \cod D_2$, a \emph{regular isomorphism of sheet diagrams} from $D_1$ to $D_2$ is given by an anchored isomorphism of progressive plane diagrams $\alpha : S(D_1) \rightarrow S(D_2)$, which gives an open graph isomorphism $\phi$ by Lemma~\ref{lemma:open-graph-iso}, as well as:
  \begin{itemize}
	  \item for each node $n_{vi}$ on a seam $v$ in $D_1$, a continuous map $x_{vi}^* : [0,1] \rightarrow [0,1]$ such that $x_{vi}^*(0) = x_{vi}$ and $x_{vi}^*(1) = x_{\phi(v)i}$, where $n_{\phi(v)i} \in D_2$;
  \item for each wire $w_{ei}$ on a sheet $e$ in $D_1$, a continuous isotopy $\gamma_{ei}^* : [0,1]\times [0,1] \rightarrow [0,1]$ such that $\gamma_{ei}^*(0,t) = \gamma_{ei}(t)$ and $\gamma_{ei}^*(1,t) = \gamma_{\phi(e)i}(t)$ for all $t \in [0,1]$, where $w_{\phi(e)i} \in D_2$
  \end{itemize}
  Finally, at each time $t \in [0, 1]$, the sheet diagram made of the skeleton $\alpha(t)$, the node positions $x_{vi}^*(t)$ and the wire paths $\gamma_{ei}^*(t)$ is required to be a valid sheet diagram.
\end{definition}

\begin{lemma} \label{lemma:regular-isomorphism-preserves-interpretation}
  Regular isomorphisms of sheet diagrams preserve their interpretations as bimonoidal morphisms.
\end{lemma}

\begin{proof}
  As the interpretation of a sheet diagram $D$ is the interpretation of its skeleton $S(D)$,
  this is a simple consequence of Theorem~\ref{thm:joyal-street}.
\end{proof}

However, regular isomorphism of sheet diagrams does not capture the
entire equational theory of bimonoidal categories: the exchange law
for the multiplicative monoidal structure is missing.
For instance, the following diagrams have the same interpretation, but
they are not regularly isomorphic:
\begin{figure}[H]
  \centering
  \begin{subfigure}{0.45\textwidth}
    \centering
    \scalebox{0.8}{
\begingroup%
  \makeatletter%
  \providecommand\color[2][]{%
    \errmessage{(Inkscape) Color is used for the text in Inkscape, but the package 'color.sty' is not loaded}%
    \renewcommand\color[2][]{}%
  }%
  \providecommand\transparent[1]{%
    \errmessage{(Inkscape) Transparency is used (non-zero) for the text in Inkscape, but the package 'transparent.sty' is not loaded}%
    \renewcommand\transparent[1]{}%
  }%
  \providecommand\rotatebox[2]{#2}%
  \newcommand*\fsize{\dimexpr\f@size pt\relax}%
  \newcommand*\lineheight[1]{\fontsize{\fsize}{#1\fsize}\selectfont}%
  \ifx\svgwidth\undefined%
    \setlength{\unitlength}{79.17790003bp}%
    \ifx\svgscale\undefined%
      \relax%
    \else%
      \setlength{\unitlength}{\unitlength * \real{\svgscale}}%
    \fi%
  \else%
    \setlength{\unitlength}{\svgwidth}%
  \fi%
  \global\let\svgwidth\undefined%
  \global\let\svgscale\undefined%
  \makeatother%
  \begin{picture}(1,1.33575491)%
    \lineheight{1}%
    \setlength\tabcolsep{0pt}%
    \put(0,0){\includegraphics[width=\unitlength,page=1]{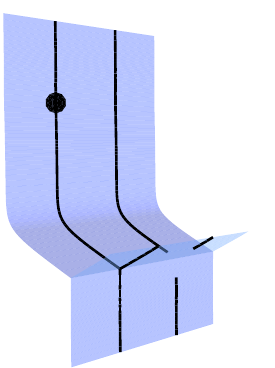}}%
    \put(0.07483519,0.92874355){\color[rgb]{0,0,0}\makebox(0,0)[t]{\lineheight{0}\smash{\begin{tabular}[t]{c}$g$\end{tabular}}}}%
    \put(0,0){\includegraphics[width=\unitlength,page=2]{nonregular-isotopy-1.pdf}}%
    \put(0.54075043,0.948664){\color[rgb]{0,0,0}\makebox(0,0)[t]{\lineheight{0}\smash{\begin{tabular}[t]{c}$g$\end{tabular}}}}%
    \put(0.55986942,0.28508246){\color[rgb]{0,0,0}\makebox(0,0)[t]{\lineheight{0}\smash{\begin{tabular}[t]{c}$f$\end{tabular}}}}%
    \put(0,0){\includegraphics[width=\unitlength,page=3]{nonregular-isotopy-1.pdf}}%
  \end{picture}%
\endgroup%

      }
  \end{subfigure}
  \begin{subfigure}{0.45\textwidth}
    \centering
    \scalebox{0.8}{
\begingroup%
  \makeatletter%
  \providecommand\color[2][]{%
    \errmessage{(Inkscape) Color is used for the text in Inkscape, but the package 'color.sty' is not loaded}%
    \renewcommand\color[2][]{}%
  }%
  \providecommand\transparent[1]{%
    \errmessage{(Inkscape) Transparency is used (non-zero) for the text in Inkscape, but the package 'transparent.sty' is not loaded}%
    \renewcommand\transparent[1]{}%
  }%
  \providecommand\rotatebox[2]{#2}%
  \newcommand*\fsize{\dimexpr\f@size pt\relax}%
  \newcommand*\lineheight[1]{\fontsize{\fsize}{#1\fsize}\selectfont}%
  \ifx\svgwidth\undefined%
    \setlength{\unitlength}{78.07809961bp}%
    \ifx\svgscale\undefined%
      \relax%
    \else%
      \setlength{\unitlength}{\unitlength * \real{\svgscale}}%
    \fi%
  \else%
    \setlength{\unitlength}{\svgwidth}%
  \fi%
  \global\let\svgwidth\undefined%
  \global\let\svgscale\undefined%
  \makeatother%
  \begin{picture}(1,1.35410315)%
    \lineheight{1}%
    \setlength\tabcolsep{0pt}%
    \put(0,0){\includegraphics[width=\unitlength,page=1]{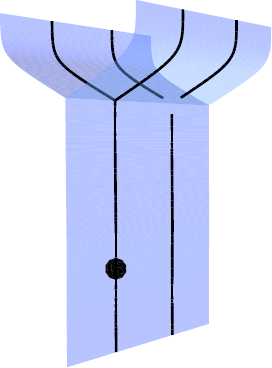}}%
    \put(0.29765171,0.32233696){\color[rgb]{0,0,0}\makebox(0,0)[t]{\lineheight{0}\smash{\begin{tabular}[t]{c}$g$\end{tabular}}}}%
    \put(0.5382199,0.88210005){\color[rgb]{0,0,0}\makebox(0,0)[t]{\lineheight{0}\smash{\begin{tabular}[t]{c}$f$\end{tabular}}}}%
    \put(0,0){\includegraphics[width=\unitlength,page=2]{nonregular-isotopy-2.pdf}}%
  \end{picture}%
\endgroup%

    }
  \end{subfigure}
\end{figure}

\noindent Their interpretations are equal by the exchange rule for the multiplicative
product:
\begin{align*}
  ((g \cdot 1_B) \oplus (g \cdot 1_C)) \circ (1_A \cdot f) &= (g \cdot (1_{B \oplus C})) \circ (1_A \cdot f) \\
  &= (1_C \cdot f) \circ (g \cdot 1_D)
\end{align*}
Therefore we need to broaden our class of isomorphisms to capture multiplicative exchange
too.

\begin{definition} \label{def:tensor-merge}
  Let $f : \bigoplus_i A_i \rightarrow \bigoplus_j B_j$ and $g : \bigoplus_k C_k \rightarrow \bigoplus_l D_l$ be morphism generators in $\Gamma$ (diagrams with a single seam), where $A_i, B_j, C_k, D_l \in \Ob \Gamma$.
  A \emph{tensor merge} from $\alpha = E_{BD}^{-1} \circ (\bigoplus_l f D_l) \circ E_{AD} \circ (\bigoplus_i A_i g)$ to $\beta = f g$
  is a function $\gamma : [0, 1] \rightarrow D(\Gamma)$ (where $D(\Sigma)$ is the set of sheet diagrams on $\Sigma$) such that:
  \begin{itemize}
  \item $\gamma(0) = \alpha$ and $\gamma(1) = \beta$;
  \item the restriction of $\gamma$ on $[0, t)$ for all $0 < t < 1$ is a regular isomorphism of sheet diagrams
  \item for each seam $s \in \alpha$, $\lim_{t \rightarrow 1} s(t)$ is the unique seam of $\beta$;
  \item for each sheet $s \in \alpha$ with one end on the boundary of the diagram, $\lim_{t \rightarrow 1} s(t)$ is the unique sheet in $\beta$ connected to the same boundary at the same ordinal position;
  \item for each sheet $s \in \alpha$ not connected to the boundary, $\lim_{t \rightarrow 1} s(t)$ is the unique seam of $\beta$;
  \item for each node $n \in \alpha$ on a seam $s$, $\lim_{t \rightarrow 1} n(t)$ is a node in the unique seam of $\beta$, with the same ordinal position;
  \item for each wire $w$ on a sheet $s \in \alpha$ that connects to the boundary, $\lim_{t \rightarrow 1} w(t)$ is a wire on $\lim_{t \rightarrow 1} s(t)$ with the same ordinal position.
  \end{itemize}
  Similarly one can define a tensor merge from $(\bigoplus_j B_j g) \circ E_{BC}^{-1} \circ (\bigoplus_k f C_k) \circ E_{AC}^{-1}$ to $fg$. Finally, a \emph{tensor explosion} $\gamma : \alpha \rightarrow \beta$ is a tensor merge in reverse, i.e. when $t \mapsto \gamma(1-t)$ is a tensor merge.
\end{definition}

For example, the following are steps of a tensor merge:
\begin{figure}[H]
  \centering
  \begin{subfigure}{0.3\textwidth}
    \centering
    \scalebox{0.8}{
      
    }
    \caption{$\gamma(0)$}
  \end{subfigure}
  \begin{subfigure}{0.3\textwidth}
    \centering
    \scalebox{0.8}{
\begingroup%
  \makeatletter%
  \providecommand\color[2][]{%
    \errmessage{(Inkscape) Color is used for the text in Inkscape, but the package 'color.sty' is not loaded}%
    \renewcommand\color[2][]{}%
  }%
  \providecommand\transparent[1]{%
    \errmessage{(Inkscape) Transparency is used (non-zero) for the text in Inkscape, but the package 'transparent.sty' is not loaded}%
    \renewcommand\transparent[1]{}%
  }%
  \providecommand\rotatebox[2]{#2}%
  \newcommand*\fsize{\dimexpr\f@size pt\relax}%
  \newcommand*\lineheight[1]{\fontsize{\fsize}{#1\fsize}\selectfont}%
  \ifx\svgwidth\undefined%
    \setlength{\unitlength}{78.60258361bp}%
    \ifx\svgscale\undefined%
      \relax%
    \else%
      \setlength{\unitlength}{\unitlength * \real{\svgscale}}%
    \fi%
  \else%
    \setlength{\unitlength}{\svgwidth}%
  \fi%
  \global\let\svgwidth\undefined%
  \global\let\svgscale\undefined%
  \makeatother%
  \begin{picture}(1,1.34519624)%
    \lineheight{1}%
    \setlength\tabcolsep{0pt}%
    \put(0,0){\includegraphics[width=\unitlength,page=1]{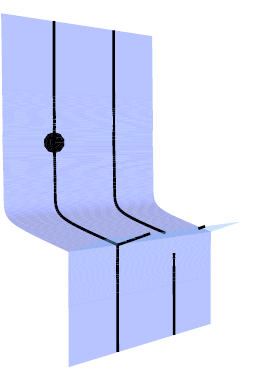}}%
    \put(0.07538293,0.78503223){\color[rgb]{0,0,0}\makebox(0,0)[t]{\lineheight{0}\smash{\begin{tabular}[t]{c}$g$\end{tabular}}}}%
    \put(0,0){\includegraphics[width=\unitlength,page=2]{tensor-merge-intermediate.pdf}}%
    \put(0.54971723,0.79768292){\color[rgb]{0,0,0}\makebox(0,0)[t]{\lineheight{0}\smash{\begin{tabular}[t]{c}$g$\end{tabular}}}}%
    \put(0.55157181,0.36536409){\color[rgb]{0,0,0}\makebox(0,0)[t]{\lineheight{0}\smash{\begin{tabular}[t]{c}$f$\end{tabular}}}}%
    \put(0,0){\includegraphics[width=\unitlength,page=3]{tensor-merge-intermediate.pdf}}%
  \end{picture}%
\endgroup%

    }
    \caption{$\gamma(\frac{1}{2})$}
  \end{subfigure}
  \begin{subfigure}{0.3\textwidth}
    \centering
    \scalebox{0.8}{
\begingroup%
  \makeatletter%
  \providecommand\color[2][]{%
    \errmessage{(Inkscape) Color is used for the text in Inkscape, but the package 'color.sty' is not loaded}%
    \renewcommand\color[2][]{}%
  }%
  \providecommand\transparent[1]{%
    \errmessage{(Inkscape) Transparency is used (non-zero) for the text in Inkscape, but the package 'transparent.sty' is not loaded}%
    \renewcommand\transparent[1]{}%
  }%
  \providecommand\rotatebox[2]{#2}%
  \newcommand*\fsize{\dimexpr\f@size pt\relax}%
  \newcommand*\lineheight[1]{\fontsize{\fsize}{#1\fsize}\selectfont}%
  \ifx\svgwidth\undefined%
    \setlength{\unitlength}{78.13968bp}%
    \ifx\svgscale\undefined%
      \relax%
    \else%
      \setlength{\unitlength}{\unitlength * \real{\svgscale}}%
    \fi%
  \else%
    \setlength{\unitlength}{\svgwidth}%
  \fi%
  \global\let\svgwidth\undefined%
  \global\let\svgscale\undefined%
  \makeatother%
  \begin{picture}(1,1.35342133)%
    \lineheight{1}%
    \setlength\tabcolsep{0pt}%
    \put(0,0){\includegraphics[width=\unitlength,page=1]{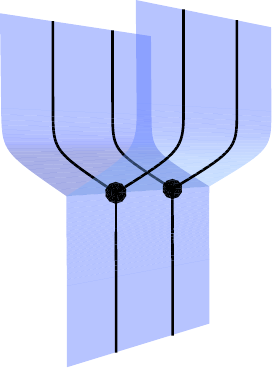}}%
    \put(0.33546903,0.5356267){\color[rgb]{0,0,0}\makebox(0,0)[t]{\lineheight{0}\smash{\begin{tabular}[t]{c}$g$\end{tabular}}}}%
    \put(0.5545723,0.55534887){\color[rgb]{0,0,0}\makebox(0,0)[t]{\lineheight{0}\smash{\begin{tabular}[t]{c}$f$\end{tabular}}}}%
  \end{picture}%
\endgroup%

    }
    \caption{$\gamma(1)$}
  \end{subfigure}
\end{figure}

We can extend the notions of tensor merges and tensor explosions to wider contexts,
where the seams to merge or explode are parts of a larger diagram.

\begin{definition}
  Let $\gamma : \alpha \rightarrow \beta$ be a tensor merge and $C(x)$ be a sheet diagram
  with a hole, such that $C(\alpha)$ is a valid sheet diagram. Since $\alpha$ and $\beta$ have
  the same domain and codomain, $C(\beta)$ is also a valid sheet diagram.
  The function $C(\gamma) : [0,1] \rightarrow C(\Sigma)$ defined by $C(\gamma) : t \mapsto C(\gamma(t))$ is called a \emph{tensor merge in context}. Similarly, we define tensor explosions in context.
\end{definition}

\begin{lemma} \label{lemma:tensor-merge-explosion-in-context}
  For all tensor merges or explosions in context $C(\gamma) : C(\alpha) \rightarrow C(\beta)$,
  $C(\alpha)$ and $C(\beta)$ are equal as bimonoidal morphisms.
\end{lemma}

\begin{proof}
  By Definition~\ref{def:tensor-merge}, the start and end diagrams of tensor merges or explosions
  are equal by multiplicative exchange. By composition, this extends to contexts.
\end{proof}

\noindent We can now define our most general notion of isotopy for sheet diagrams.

\begin{definition} \label{def:bimonoidal-isotopy}
  A \emph{bimonoidal isotopy} between sheet diagrams $D_1, D_2$ is a function $\gamma : [0,1] \rightarrow D(\Sigma)$ such that $\gamma(0) = D_1$, $\gamma(1) = D_2$ and for all $t \in [0,1]$, there exists $\varepsilon > 0$ such that on $[t-\varepsilon, t]$, $\gamma$ is either a regular isomorphism or a tensor merge, and on $[t, t+\varepsilon]$, $\gamma$ is either a regular isomorphism or a tensor explosion.
\end{definition}

\begin{lemma} \label{lemma:isotopy-preserves-interpretation}
  Bimonoidal isotopy preserves the interpretation of diagrams.
\end{lemma}

\begin{proof}
  Since $[0,1]$ is connected, it is enough to show that the interpretation of $\gamma(t)$ is locally constant for all $t \in [0,1]$.
  By Lemma~\ref{lemma:regular-isomorphism-preserves-interpretation}, the interpretation is constant during regular isomorphisms
  and by Lemma~\ref{lemma:tensor-merge-explosion-in-context}, tensor merges and explosions in context also preserve interpretation.
\end{proof}

\begin{lemma} \label{lemma:compositions-respect-isotopy}
  Composition, sum and tensor of sheet diagrams all respect bimonoidal isotopy.
\end{lemma}

\begin{proof}
  Given two sheet diagrams $\alpha, \beta$ and bimonoidal isotopies $\gamma : \alpha \rightarrow \alpha'$, $\gamma' : \beta \rightarrow \beta'$, we obtain a bimonoidal isotopy from $\alpha \oplus \beta$ to $\alpha' \oplus \beta'$ by first running $\gamma$ while $\beta$ stays still, then running $\gamma'$ while $\alpha'$ stays still. Note that we do not run both transformations in parallel because our definition of bimonoidal isotopy only allows for one tensor merge or explosion at a time. The case for the composition of diagrams is similar. By Definition~\ref{def:sheet-diagram-tensor}, the diagram $\alpha \otimes \beta$ contains in general multiple copies of $\alpha$ and $\beta$: we obtain an isotopy by running $\gamma$ on each copy of $\alpha$ in sequence, and then $\gamma'$ on copies of $\beta$.
\end{proof}

\begin{definition}
  The \emph{category of sheet diagrams} $D(\Sigma)$ has sums of products of object symbols from $\Sigma$ as objects, and equivalence classes of sheet diagrams under sheet diagram isotopy as morphisms. Domains and codomains are given by Definition~\ref{def:sheet-diagram-domain}, composition by Definition~\ref{def:sheet-diagram-composition}.
  It has a symmetric monoidal structure $\oplus$ given by Definition~\ref{def:sheet-diagram-sum}.
\end{definition}

To equip $D(\Sigma)$ with a multiplicative monoidal structure, we need to show that our tensor product
(Definition~\ref{def:sheet-diagram-tensor}) satisfies the exchange law.
Tensor merges and explosions are only defined for morphism generators in $\Gamma$ (single seams), not arbitrary diagrams, so we cannot just use one tensor merge followed by one tensor explosion in general.

\begin{lemma} \label{lemma:slice-decomposition}
  Any diagram $f \in D(\Sigma)$ can be written in general position,
  such that no two seams or additive symmetries are at the same
  height, up to bimonoidal isotopy. It can therefore be expressed as a
  sequential composition of \emph{slices}, which are sums of at most
  one seam or additive symmetry and a finite number of identities.
\end{lemma}

\begin{proof}
  This is a straightforward generalization of the same result for
  symmetric monoidal string diagrams, which can be found in \cite{joyal1991geometry-1}.
\end{proof}

\begin{lemma} \label{lemma:tensor-slice}
  Let $f : \bigoplus_i A_i \rightarrow \bigoplus_j B_j$ and $g : \bigoplus_k C_k \rightarrow \bigoplus_l D_l$ be slices, where $A_i, B_j, C_k, D_l \in \Ob \Gamma$.
  Then there is a bimonoidal isotopy between $E_{BD}^{-1} \circ (\bigoplus_l f D_l) \circ E_{AD} \circ (\bigoplus_i A_i g)$ and $(\bigoplus_j B_j g) \circ E^{-1}_{BC} \circ (\bigoplus_k f C_k) \circ E^{-1}_{AC}$.
\end{lemma}

\begin{proof}
  We proceed by induction on the sum of numbers of identities in the summands of $f$ and $g$.
  When there are no identities in $f$ or $g$, there are three cases. If both $f$ and $g$ are seams, then the two expressions are isotopic via a tensor merge and explosion by construction. If both $f$ and $g$ are symmetries, then the two expressions only consist of symmetries and identities which induce the same permutation of the summands of their domain, so they are isotopic. Finally if one of $f$ and $g$ is a seam and the other is a symmetry, let us assume by symmetry that $f = \gamma_{A_1,A_2}$ and $g$ is a seam. The isotopy holds by pull through moves:

  \begin{figure}[H]
    \centering
    \begin{tikzpicture}[every node/.style={scale=0.8}]
      \node at (-2.5,0) {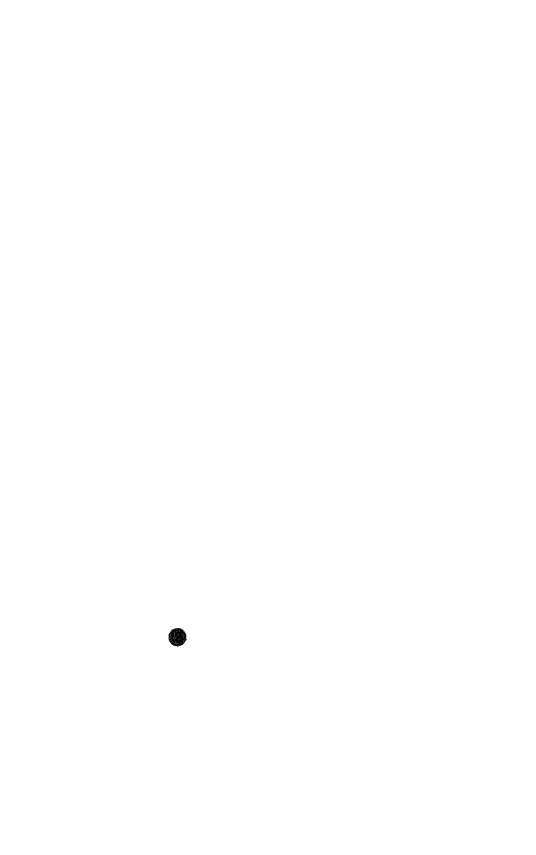};
      \node at (0,0) {\Large $=$};
      \node at (2.5,0) {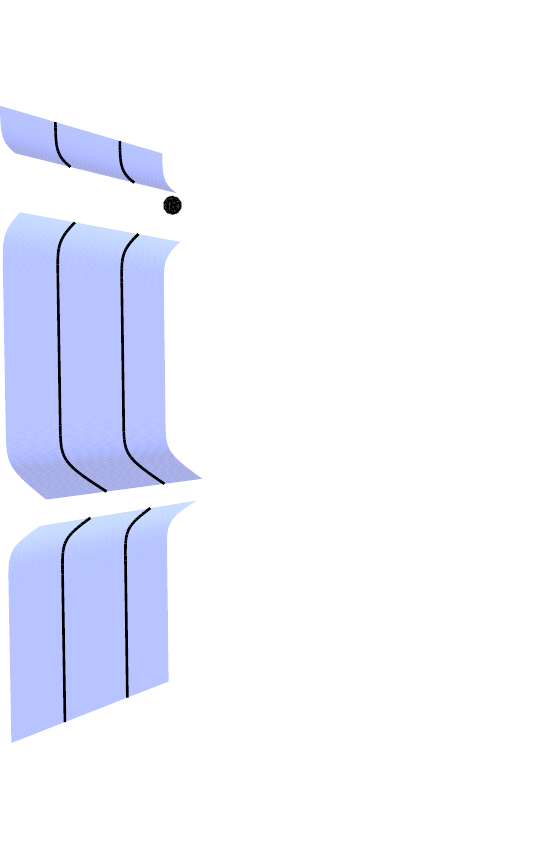};
    \end{tikzpicture}
  \end{figure}
  \noindent Notice that in this transformation nothing interesting is happening on the third dimension: in the future, we will resort to two dimensional string diagrams for such isotopies.
  
  Now for the inductive case, assume there is an identity in $f = 1_{A_1} \oplus f'$. The isotopy holds
  as follows:
  \begin{figure}[H]
    \centering
    \begin{tikzpicture}
    \node at (0,0) {\def\svgscale{0.7}
    \small
\begingroup%
  \makeatletter%
  \providecommand\color[2][]{%
    \errmessage{(Inkscape) Color is used for the text in Inkscape, but the package 'color.sty' is not loaded}%
    \renewcommand\color[2][]{}%
  }%
  \providecommand\transparent[1]{%
    \errmessage{(Inkscape) Transparency is used (non-zero) for the text in Inkscape, but the package 'transparent.sty' is not loaded}%
    \renewcommand\transparent[1]{}%
  }%
  \providecommand\rotatebox[2]{#2}%
  \newcommand*\fsize{\dimexpr\f@size pt\relax}%
  \newcommand*\lineheight[1]{\fontsize{\fsize}{#1\fsize}\selectfont}%
  \ifx\svgwidth\undefined%
    \setlength{\unitlength}{142.04853191bp}%
    \ifx\svgscale\undefined%
      \relax%
    \else%
      \setlength{\unitlength}{\unitlength * \real{\svgscale}}%
    \fi%
  \else%
    \setlength{\unitlength}{\svgwidth}%
  \fi%
  \global\let\svgwidth\undefined%
  \global\let\svgscale\undefined%
  \makeatother%
  \begin{picture}(1,1.26717255)%
    \lineheight{1}%
    \setlength\tabcolsep{0pt}%
    \put(0,0){\includegraphics[width=\unitlength,page=1]{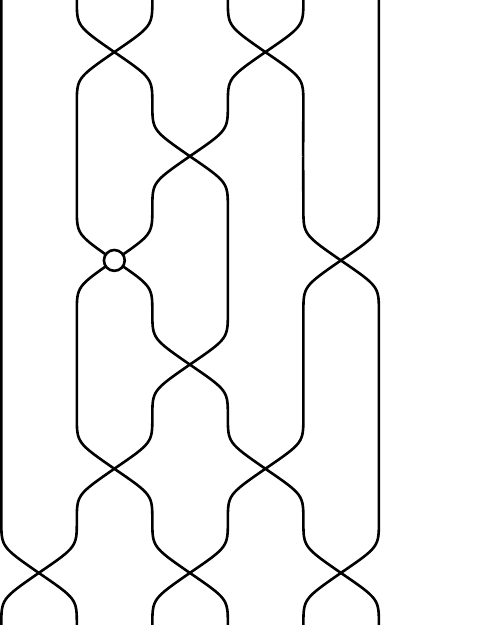}}%
    \put(0.28458585,0.73918399){\makebox(0,0)[lt]{\lineheight{1.25}\smash{\begin{tabular}[t]{l}$f' D_1$\end{tabular}}}}%
    \put(0,0){\includegraphics[width=\unitlength,page=2]{slice-exchange-1.pdf}}%
    \put(0.74393586,0.73918399){\makebox(0,0)[lt]{\lineheight{1.25}\smash{\begin{tabular}[t]{l}$f' D_2$\end{tabular}}}}%
    \put(0,0){\includegraphics[width=\unitlength,page=3]{slice-exchange-1.pdf}}%
    \put(0.13146917,0.10559771){\makebox(0,0)[lt]{\lineheight{1.25}\smash{\begin{tabular}[t]{l}$A_1 g$\end{tabular}}}}%
    \put(0,0){\includegraphics[width=\unitlength,page=4]{slice-exchange-1.pdf}}%
    \put(0.43770253,0.10559771){\makebox(0,0)[lt]{\lineheight{1.25}\smash{\begin{tabular}[t]{l}$A_2 g$\end{tabular}}}}%
    \put(0,0){\includegraphics[width=\unitlength,page=5]{slice-exchange-1.pdf}}%
    \put(0.74393586,0.10559771){\makebox(0,0)[lt]{\lineheight{1.25}\smash{\begin{tabular}[t]{l}$A_3 g$\end{tabular}}}}%
  \end{picture}%
\endgroup%
};
    \node at (2.5, 0) {\Large $=$};
    \node at (5,0) {\def\svgscale{0.7}
    \small
\begingroup%
  \makeatletter%
  \providecommand\color[2][]{%
    \errmessage{(Inkscape) Color is used for the text in Inkscape, but the package 'color.sty' is not loaded}%
    \renewcommand\color[2][]{}%
  }%
  \providecommand\transparent[1]{%
    \errmessage{(Inkscape) Transparency is used (non-zero) for the text in Inkscape, but the package 'transparent.sty' is not loaded}%
    \renewcommand\transparent[1]{}%
  }%
  \providecommand\rotatebox[2]{#2}%
  \newcommand*\fsize{\dimexpr\f@size pt\relax}%
  \newcommand*\lineheight[1]{\fontsize{\fsize}{#1\fsize}\selectfont}%
  \ifx\svgwidth\undefined%
    \setlength{\unitlength}{144.74853086bp}%
    \ifx\svgscale\undefined%
      \relax%
    \else%
      \setlength{\unitlength}{\unitlength * \real{\svgscale}}%
    \fi%
  \else%
    \setlength{\unitlength}{\svgwidth}%
  \fi%
  \global\let\svgwidth\undefined%
  \global\let\svgscale\undefined%
  \makeatother%
  \begin{picture}(1,0.82902396)%
    \lineheight{1}%
    \setlength\tabcolsep{0pt}%
    \put(0,0){\includegraphics[width=\unitlength,page=1]{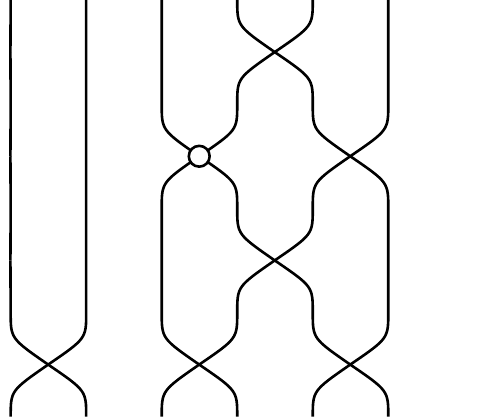}}%
    \put(0.44819109,0.51813997){\makebox(0,0)[lt]{\lineheight{1.25}\smash{\begin{tabular}[t]{l}$f' D_1$\end{tabular}}}}%
    \put(0,0){\includegraphics[width=\unitlength,page=2]{slice-exchange-2.pdf}}%
    \put(0.74871223,0.51813997){\makebox(0,0)[lt]{\lineheight{1.25}\smash{\begin{tabular}[t]{l}$f' D_2$\end{tabular}}}}%
    \put(0,0){\includegraphics[width=\unitlength,page=3]{slice-exchange-2.pdf}}%
    \put(0.1476699,0.10362799){\makebox(0,0)[lt]{\lineheight{1.25}\smash{\begin{tabular}[t]{l}$A_1 g$\end{tabular}}}}%
    \put(0,0){\includegraphics[width=\unitlength,page=4]{slice-exchange-2.pdf}}%
    \put(0.44819109,0.10362799){\makebox(0,0)[lt]{\lineheight{1.25}\smash{\begin{tabular}[t]{l}$A_2 g$\end{tabular}}}}%
    \put(0,0){\includegraphics[width=\unitlength,page=5]{slice-exchange-2.pdf}}%
    \put(0.74871223,0.10362799){\makebox(0,0)[lt]{\lineheight{1.25}\smash{\begin{tabular}[t]{l}$A_3 g$\end{tabular}}}}%
  \end{picture}%
\endgroup%
};
        \node at (-2.5, -5) {\Large $=$};
    \node at (0,-5) {\def\svgscale{0.7}
    \small
\begingroup%
  \makeatletter%
  \providecommand\color[2][]{%
    \errmessage{(Inkscape) Color is used for the text in Inkscape, but the package 'color.sty' is not loaded}%
    \renewcommand\color[2][]{}%
  }%
  \providecommand\transparent[1]{%
    \errmessage{(Inkscape) Transparency is used (non-zero) for the text in Inkscape, but the package 'transparent.sty' is not loaded}%
    \renewcommand\transparent[1]{}%
  }%
  \providecommand\rotatebox[2]{#2}%
  \newcommand*\fsize{\dimexpr\f@size pt\relax}%
  \newcommand*\lineheight[1]{\fontsize{\fsize}{#1\fsize}\selectfont}%
  \ifx\svgwidth\undefined%
    \setlength{\unitlength}{142.04853191bp}%
    \ifx\svgscale\undefined%
      \relax%
    \else%
      \setlength{\unitlength}{\unitlength * \real{\svgscale}}%
    \fi%
  \else%
    \setlength{\unitlength}{\svgwidth}%
  \fi%
  \global\let\svgwidth\undefined%
  \global\let\svgscale\undefined%
  \makeatother%
  \begin{picture}(1,0.8447817)%
    \lineheight{1}%
    \setlength\tabcolsep{0pt}%
    \put(0,0){\includegraphics[width=\unitlength,page=1]{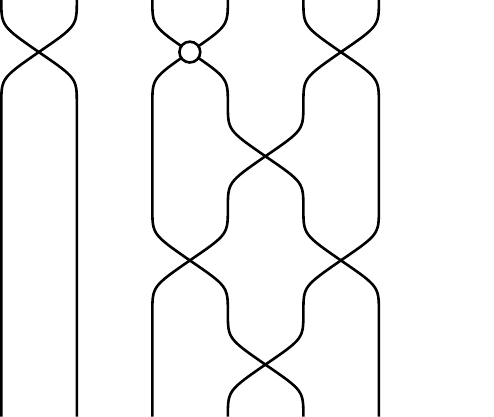}}%
    \put(0.43770253,0.73918399){\makebox(0,0)[lt]{\lineheight{1.25}\smash{\begin{tabular}[t]{l}$A_2 g$\end{tabular}}}}%
    \put(0,0){\includegraphics[width=\unitlength,page=2]{slice-exchange-3.pdf}}%
    \put(0.74393586,0.73918399){\makebox(0,0)[lt]{\lineheight{1.25}\smash{\begin{tabular}[t]{l}$A_3 g$\end{tabular}}}}%
    \put(0,0){\includegraphics[width=\unitlength,page=3]{slice-exchange-3.pdf}}%
    \put(0.13146917,0.73918399){\makebox(0,0)[lt]{\lineheight{1.25}\smash{\begin{tabular}[t]{l}$A_1 g$\end{tabular}}}}%
    \put(0,0){\includegraphics[width=\unitlength,page=4]{slice-exchange-3.pdf}}%
    \put(0.43770253,0.31679314){\makebox(0,0)[lt]{\lineheight{1.25}\smash{\begin{tabular}[t]{l}$f' D_1$\end{tabular}}}}%
    \put(0,0){\includegraphics[width=\unitlength,page=5]{slice-exchange-3.pdf}}%
    \put(0.74393586,0.31679314){\makebox(0,0)[lt]{\lineheight{1.25}\smash{\begin{tabular}[t]{l}$f' D_2$\end{tabular}}}}%
    \put(0,0){\includegraphics[width=\unitlength,page=6]{slice-exchange-3.pdf}}%
  \end{picture}%
\endgroup%
};
        \node at (2.5, -5) {\Large $=$};
    \node at (5,-5) {\def\svgscale{0.7}
    \small
\begingroup%
  \makeatletter%
  \providecommand\color[2][]{%
    \errmessage{(Inkscape) Color is used for the text in Inkscape, but the package 'color.sty' is not loaded}%
    \renewcommand\color[2][]{}%
  }%
  \providecommand\transparent[1]{%
    \errmessage{(Inkscape) Transparency is used (non-zero) for the text in Inkscape, but the package 'transparent.sty' is not loaded}%
    \renewcommand\transparent[1]{}%
  }%
  \providecommand\rotatebox[2]{#2}%
  \newcommand*\fsize{\dimexpr\f@size pt\relax}%
  \newcommand*\lineheight[1]{\fontsize{\fsize}{#1\fsize}\selectfont}%
  \ifx\svgwidth\undefined%
    \setlength{\unitlength}{142.04853191bp}%
    \ifx\svgscale\undefined%
      \relax%
    \else%
      \setlength{\unitlength}{\unitlength * \real{\svgscale}}%
    \fi%
  \else%
    \setlength{\unitlength}{\svgwidth}%
  \fi%
  \global\let\svgwidth\undefined%
  \global\let\svgscale\undefined%
  \makeatother%
  \begin{picture}(1,1.26717255)%
    \lineheight{1}%
    \setlength\tabcolsep{0pt}%
    \put(0,0){\includegraphics[width=\unitlength,page=1]{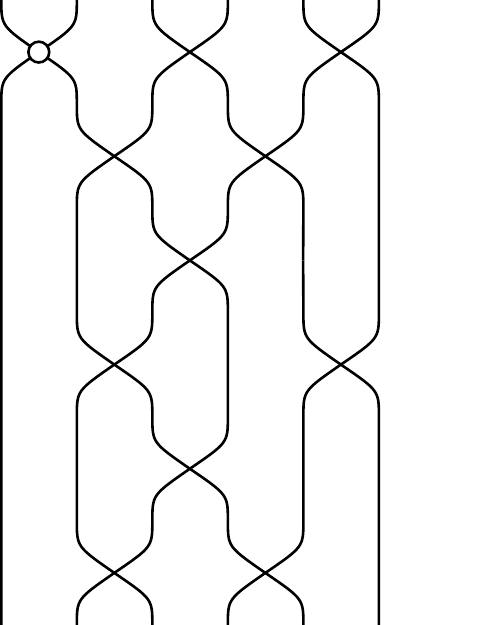}}%
    \put(0.13146917,1.16157483){\makebox(0,0)[lt]{\lineheight{1.25}\smash{\begin{tabular}[t]{l}$A_1 g$\end{tabular}}}}%
    \put(0,0){\includegraphics[width=\unitlength,page=2]{slice-exchange-4.pdf}}%
    \put(0.43770253,1.16157483){\makebox(0,0)[lt]{\lineheight{1.25}\smash{\begin{tabular}[t]{l}$A_2 g$\end{tabular}}}}%
    \put(0,0){\includegraphics[width=\unitlength,page=3]{slice-exchange-4.pdf}}%
    \put(0.74393586,1.16157483){\makebox(0,0)[lt]{\lineheight{1.25}\smash{\begin{tabular}[t]{l}$A_3 g$\end{tabular}}}}%
    \put(0,0){\includegraphics[width=\unitlength,page=4]{slice-exchange-4.pdf}}%
    \put(0.28458585,0.52798856){\makebox(0,0)[lt]{\lineheight{1.25}\smash{\begin{tabular}[t]{l}$f' D_1$\end{tabular}}}}%
    \put(0,0){\includegraphics[width=\unitlength,page=5]{slice-exchange-4.pdf}}%
    \put(0.74393586,0.52798856){\makebox(0,0)[lt]{\lineheight{1.25}\smash{\begin{tabular}[t]{l}$f' D_2$\end{tabular}}}}%
    \put(0,0){\includegraphics[width=\unitlength,page=6]{slice-exchange-4.pdf}}%
  \end{picture}%
\endgroup%
};
    \end{tikzpicture}
  \end{figure}
  \noindent The second equality uses the induction hypothesis on $f'$ and $g$,
  other steps are regular isotopies.
\end{proof}

\begin{lemma} \label{lemma:tensor-recursion}
  Let $f : \bigoplus_i A_i \rightarrow \bigoplus_j B_j$ and $g : \bigoplus_k C_k \rightarrow \bigoplus_l D_l$ be sheet diagrams, where $A_i, B_j, C_k, D_l \in \Ob \Gamma$.
  Then there is a bimonoidal isotopy between $E_{BD}^{-1} \circ (\bigoplus_l f D_l) \circ E_{AD} \circ (\bigoplus_i A_i g)$ and $(\bigoplus_j B_j g) \circ E^{-1}_{BC} \circ (\bigoplus_k f C_k) \circ E^{-1}_{AC}$.
\end{lemma}

\begin{proof}
  Up to a regular isotopy, we can assume that $f$ and $g$ are in general position and therefore expressed as a sequential composition of slices. We can then apply Lemma~\ref{lemma:tensor-slice} repeatedly, exchanging neighbouring slices of $f$ and $g$ until all slices of $f$ are below all slices of $g$.
\end{proof}

\begin{lemma} \label{lemma:mult-monoidal-structure}
  $D(\Sigma)$ can be equipped with a monoidal structure $(D(\Sigma), \otimes, I)$, given on objects by $A \otimes B = N(A \cdot B)$ and on morphisms by Definition~\ref{def:sheet-diagram-tensor}
\end{lemma}

\begin{proof}
  The product on objects is unital ($N(A \cdot I) = N(I \cdot A) = N(A)$) and associative by Lemma~\ref{lemma:associativity-n}.
  By Lemma~\ref{lemma:tensor-recursion}, the exchange law for $\otimes$ is satisfied, hence the result.
\end{proof}

\begin{lemma}
  $D(\Sigma)$ is bimonoidal with $(\otimes, I)$ distributing over $(\oplus, O)$ .
\end{lemma}

\begin{proof}
  The monoidal structure $(D(\Sigma), \oplus, O)$ is given by
  Theorem~\ref{thm:joyal-street}, and the monoidal structure
  $(D(\Sigma), \otimes, I)$ is given by
  Lemma~\ref{lemma:mult-monoidal-structure}.

  Since $N((A \oplus B)C) = N(AC) \oplus N(BC)$, we can define:
  $$\delta^{\#}_{A,B,C} : (A \oplus B) \otimes C \rightarrow A \otimes B \oplus B \otimes C = 1_{N(AC) \oplus N(BC)}$$
  
  For $\delta_{A,B,C} : A \otimes (B \oplus C) \rightarrow A \otimes B \oplus A \otimes C$, decompose $N(A) = \bigoplus_i A_i$ with $A_i \in \Ob \Gamma$.
  We have $$N(A(B \oplus C)) = \bigoplus_i N(A_i(B \oplus C)) = \bigoplus_i N(A_iB) \oplus N(A_i C)$$
  $$N(AC) \oplus N(BC) = \bigoplus_i N(A_i B) \oplus \bigoplus_i N(A_i C)$$
  Therefore we define $\delta_{A,B,C}$ as the reordering map (composition of symmetries) from $\bigoplus_i N(A_iB) \oplus N(A_iC)$
  to $\bigoplus_i N(A_i B) \oplus \bigoplus_i N(A_i C)$.

  Since $N(OA) = N(AO) = O$ for all $A$, we define $\lambda^* : O \otimes A \rightarrow O$ and
  $\rho_A^* : AO \rightarrow O$ as $1_O$ (the empty sheet diagram).

  We now need to check the coherence axioms of Appendix~\ref{app:coherence-axioms}. Let us consider the first axiom:
\[
\begin{tikzcd}[column sep=small]
A(B \oplus C) \arrow[rr, "{\delta_{A,B,C}}"] \arrow[dd, "{1_A \gamma'_{B,C}}"'] &            & AB \oplus AC \arrow[dd, "{\gamma'_{AB,AC}}"] \\
                                                                                & \text{(I)} &                                              \\
A(C\oplus B) \arrow[rr, "{\delta_{A,C,B}}"']                                    &            & AC \oplus AB                                
\end{tikzcd}
\]
In $D(\Sigma)$, all sides of this square are composites of the
additive symmetry $\gamma$.  Therefore it is sufficient to check that
both paths induce the same permutation, by coherence for symmetric
monoidal categories. Equivalently, one can also use coherence for
regular objects of bimonoidal categories
(Theorem~\ref{thm:regular-coherence}) by choosing $A$, $B$ and $C$
as sums of generators all distinct. One can then obtain commutation
for the general case by instantiation (substituting the generators by
the actual summands). The other axioms can be treated in similar ways:

  \begin{itemize}
  \item (I), (V), (VI), (VIII), (IX) hold by bimonoidal coherence for regular objects;
  \item (III) holds since $\delta^{\#} = 1$ and $\gamma_{A,B} 1_C = \gamma'_{AC,BC}$ by definition;
  \item (IV) simplifies thanks to $\delta^{\#} = 1$ and holds by monoidal coherence for $(\oplus, O)$
  \item (VII) simplifies thanks to $\delta^{\#} = 1$ and holds by monoidal coherence for $(\otimes, I)$
  \item (X), (XII), (XIII), (XIV), (XVI), (XVII), (XVIII) hold as all sides equal $1_O$
  \item (XIX), (XX), (XXI), (XXII) hold as all sides are identities
  \item (XXIII) and (XXIV) hold as $\delta_{I,A,B} = 1_{A \oplus B} = \delta^{\#}_{A,B,I}$
  \end{itemize}
\end{proof}

\begin{theorem} \label{thm:equiv-free-diagrams}
  $D(\Sigma)$ and $\overline{\Sigma}$ are bimonoidally equivalent, i.e. $D(\Sigma)$ is the free bimonoidal category on $\Sigma$.
\end{theorem}

\begin{proof}
  The interpretation of diagrams is a well-defined function $[ \cdot ] : D(\Sigma) \rightarrow \overline{\Sigma}$ by Lemma~\ref{lemma:isotopy-preserves-interpretation} and is a bimonoidal functor by construction.
  For the reverse direction, by freeness of $\overline{\Sigma}$ there is a unique bimonoidal functor $F : \overline{\Sigma} \rightarrow D(\Sigma)$
  mapping each generator in $\Sigma$ to its representation in $D(\Sigma)$.
  \[
  \begin{tikzpicture}
    \node[inner sep=8pt] at (0,0) (a) {$\overline{\Sigma}$};
    \node at (2,0) (b) {$D(\Sigma)$};
    \draw[->,bend left] (a) edge node[above] {$F$} (b);
    \draw[->,bend left] (b) edge node[below] {$[ \cdot ]$} (a);
  \end{tikzpicture}
  \]
  Let us show that these form an equivalence.
  First, $F \circ [ \cdot ]$ is the identity on objects and morphism generators
  and therefore by induction it is the identity on all morphisms.
  Second, $[ \cdot ] \circ F$ is not the identity but $n_A : A \rightarrow N(A)$
  is a natural isomorphism from the identity to $[ \cdot ] \circ F$.
  Its naturality can be shown by induction on $f$:
  \[
\begin{tikzcd}
A \arrow[r, "f"] \arrow[d, "n_A"'] & B \arrow[d, "n_B"] \\
N(A) \arrow[r, "F(f)"']            & N(B)              
\end{tikzcd}
\]
For $f$ a generator, the vertical sides are identities (by assumption
that domains and codomains of generators are normalized, from
Section~\ref{sec:bimonoidal-signatures}). For $f$ a structural
isomorphism, the square commutes by regular coherence.  By induction,
it holds for all morphisms.
\end{proof}

\section{Baez's conjecture}

Recently, a conjecture attributed to Baez was confirmed by
\cite{elgueta2020groupoid}, who showed that the groupoid of finite
sets and bijections is biinitial in the 2-category of bimonoidal categories.  The
category of finite sets has indeed a bimonoidal structure, where disjoint union
of sets is the monoidal addition and cartesian product is the monoidal
multiplication.

This result can also be obtained via string diagrams. Indeed, the free
bimonoidal category on an empty signature, $\overline{\emptyset}$, is biinitial.
This is a direct consequence of the universal property: any bimonoidal
functor from $\overline{\emptyset}$ to a bimonoidal category $\mathcal{C}$ is
determined (up to equivalence) by the image of the generators of
$\overline{\emptyset}$, but there are no such generators.

Therefore, to prove Baez' conjecture it is enough to characterize the
free bimonoidal category on an empty signature. By
Theorem~\ref{thm:equiv-free-diagrams}, $\overline{\emptyset}$ is
bimonoidally equivalent to $D(\emptyset)$.  The objects of
$D(\emptyset)$ are finite sums of the multiplicative monoidal
unit. The morphisms of $D(\emptyset)$ are string diagrams on the empty
signature. We can analyze the geometry of such string diagrams.  All
the sheets in such diagrams are empty: they cannot have any wires on
them, since those wires would need to be annotated by an object from
the signature.  Similarly, the diagrams do not contain any seams,
since each seam contains at least one node which is annotated by a
morphism generator. Therefore, the diagrams are only made of identities
of empty sheets and additive symmetries.

\begin{figure}[H]
    \centering
    \scalebox{0.8}{
      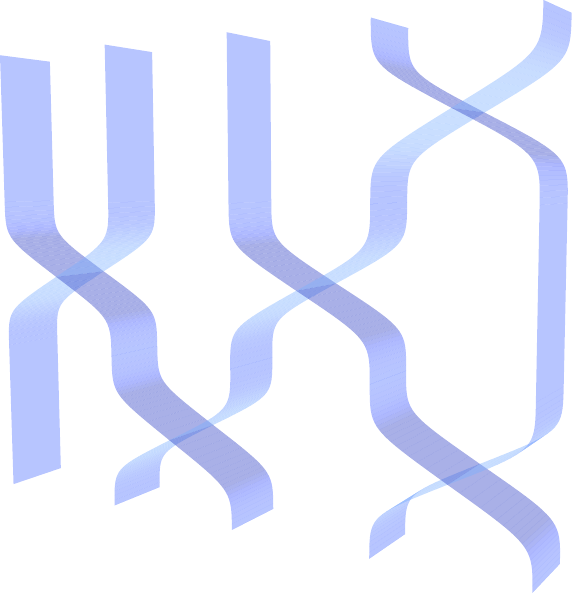
    }
    \caption{A sheet diagram on the empty signature}
\end{figure}

Hence, a morphism in $D(\emptyset)$ induces a permutation of its domain,
which is equal to its codomain. Furthermore, morphisms in $D(\emptyset)$
are equivalence classes of string diagrams up to bimonoidal equivalence.
Two string diagrams induce the same permutation of their common domain
when their skeletons are isomorphic as symmetric monoidal diagrams,
and therefore when the diagrams are in bimonoidal equivalence. Therefore,
the category $D(\emptyset)$ is equivalent to the groupoid of finite sets,
hence the result.

\bibliographystyle{plainnat}
\bibliography{references}

\appendix

\section{Data structures for sheet diagrams} \label{app:sheetshow}

In this section we give a short primer on the declarative format used to represent
diagrams in SheetShow\footnote{Available at \url{https://wetneb.github.io/sheetshow/} and \url{https://github.com/wetneb/sheetshow} (source)}, the tool we used to render all sheet diagrams in this article.

The first step to understand the format is to explain the data
structure we use to represent monoidal string diagrams.
These are simpler than bimonoidal diagrams as they can be drawn in the plane.

\subsection{Data structure for monoidal string diagrams%
  \label{data-structure-for-monoidal-string-diagrams}%
}

A string diagrams in a monoidal category is in \emph{general position} when no two nodes share the same height. Any string diagram can be put in general position without changing its meaning.
A diagram in general position can be decomposed as a sequence of horizontal \emph{slices}, each of which contains exactly one node.
\begin{figure}[H]
  \centering
  \def\svgscale{0.65}
  \small
  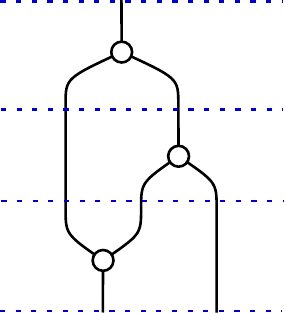
\end{figure}
\noindent One can therefore encode such a diagram as follows:
\begin{itemize}
\item The number of wires crossing the input boundary (or a list of the objects annotating them, in a typed context);
\item The list of slices, each of which can be described by the following data:
\begin{itemize}
\item The number of wires passing to the left of the node in the slice. We call this the \emph{offset};
\item The number of input wires consumed by the node;
\item The number of output wires produced by the node (or again, their list of types).
\end{itemize}
\end{itemize}

\noindent For the sample diagram above, this gives us the following encoding (with inputs at the bottom of the diagram):
\begin{verbatim}
inputs: 2
slices:
- offset: 0
  inputs: 1
  outputs: 2
- offset: 1
  inputs: 2
  outputs: 1
- offset: 0
  inputs: 2
  outputs: 1
\end{verbatim}

Each slice can be augmented to store details about the morphism in that slice (such as a label or types, for instance).
This data structure is well suited to reason about string diagrams and there are efficient algorithms to determine if two diagrams are equivalent up to exchanges. For more details about this data structure, see~\cite{delpeuch2018normalization-1}.

\subsection{Bimonoidal diagrams}

Sheet diagrams in bimonoidal categories are obtained by extruding symmetric monoidal string diagrams for
the additive monoidal structure $(\mathcal{C}, \oplus, O)$.
Therefore our data structure for bimonoidal diagrams is based on that for monoidal diagrams.


A bimonoidal diagram is described by:
\begin{itemize}
\item The number of input sheets, and the number of input wires on each of these input sheets;
\item The slices of the bimonoidal diagram, which are seams between sheets. They are each described by:
\begin{itemize}
\item The number of sheets passing to the left of the seam. We call this, again, the \emph{offset};
\item The number of input sheets joined by the seam;
\item The number of output sheets produced by the seam;
\item The nodes present on the seam.
\end{itemize}
\end{itemize}

Each seam can have multiple nodes on it. Each of these can connect to some wires on each input sheet
(not necessarily the same number of wires for each input sheet) and similarly for output sheets.
We describe them with the following data:
\begin{itemize}
\item The number of wires passing through the seam without touching a node, to the left of the node being
described. We call this the \emph{offset} of the node;
\item For each input sheet, the number of wires connected to the node;
\item For each output sheet, the number of wires connected to the node.
\end{itemize}
For instance:
\begin{figure}[H]
  \centering
  \begin{subfigure}{0.4\textwidth}
    \centering
\begin{verbatim}
inputs: [ 1, 2, 2 ]
slices:
- offset: 1
  inputs: 1
  outputs: 2
  nodes:
  - offset: 0
    inputs: [ 1 ]
    outputs: [ 1, 1 ]
- offset: 2
  inputs: 2
  outputs: 2
  nodes:
  - offset: 0
    inputs: [ 2, 2 ]
    outputs: [ 1, 1 ]
\end{verbatim}
  \end{subfigure}
  \begin{subfigure}{0.4\textwidth}
      \centering
  \def\svgscale{0.65}
  \small
  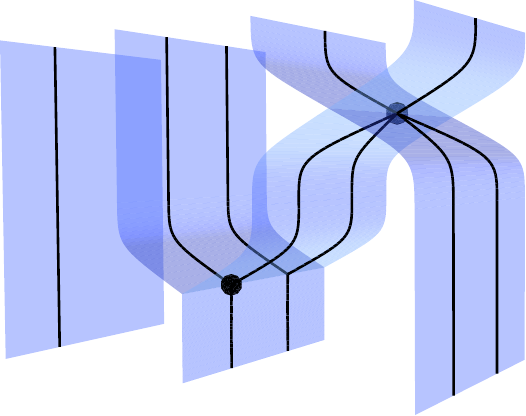
  \end{subfigure}
\end{figure}

Again, additional metadata can be added on the geometry
to annotate it with labels, types, and represent symmetries for
the additive and multiplicative structures. For more details
about these, consult SheetShow's documentation:
\url{https://sheetshow.readthedocs.org/en/latest/}.

\section{Coherence axioms}
\label{app:coherence-axioms}

The following coherence axioms are taken from \citet{laplaza1972coherence}. Their axioms (II) and (XV) were removed as they only apply to symmetric bimonoidal categories.
In the following, the monoidal structure $(\mathcal{C}, \cdot, I)$ has coherence isomorphisms
\begin{align*}
  \alpha_{A,B,C} & : A(BC) \rightarrow (AB)C \\
  \lambda_A &: IA \rightarrow A \\
  \rho_A &: AI \rightarrow A
\end{align*}
and the monoidal structure $(\mathcal{C}, \oplus, O)$ has coherence isomorphisms
\begin{align*}
  \alpha'_{A,B,C} &: A\oplus (B\oplus C) \rightarrow (A\oplus B) \oplus C \\
  \lambda'_A &: O \oplus A \rightarrow A \\
  \rho'_A &: A \oplus O \rightarrow A \\
  \gamma'_{A,B} &: A \oplus B \rightarrow B \oplus A
\end{align*}

\begin{figure}[H]
  \centering
  \begin{subfigure}{0.49\textwidth}
    \centering

\begin{tikzcd}[column sep=small]
A(B \oplus C) \arrow[rr, "{\delta_{A,B,C}}"] \arrow[dd, "{1_A \gamma'_{B,C}}"'] &            & AB \oplus AC \arrow[dd, "{\gamma'_{AB,AC}}"] \\
                                                                                & \text{(I)} &                                              \\
A(C\oplus B) \arrow[rr, "{\delta_{A,C,B}}"']                                    &            & AC \oplus AB                                
\end{tikzcd}
  \end{subfigure}
  \begin{subfigure}{0.49\textwidth}
    \centering

\begin{tikzcd}[column sep=small]
(A \oplus B) C \arrow[rr, "{\delta^{\#}_{A,B,C}}"] \arrow[dd, "{\gamma'_{A,B} 1_C}"] &              & AC \oplus BC \arrow[dd, "{\gamma'_{AC,BC}}"] \\
                                                                                     & \text{(III)} &                                              \\
(B \oplus A)C \arrow[rr, "{\delta^{\#}_{B,A,C}}"']                                   &              & BC \oplus AC                                
\end{tikzcd}
  \end{subfigure}
  \begin{subfigure}{0.99\textwidth}
\centering

\begin{tikzcd}
(A \oplus (B \oplus C))D \arrow[dd, "{\alpha'_{A,B,C} 1_D}"'] \arrow[rr, "{\delta^{\#}_{A,B\oplus C,D}}"] &  & AD \oplus (B \oplus C)D \arrow[rr, "{1_{AD} \oplus \delta^{\#}_{B,C,D}}"] &  & AD \oplus (BD \oplus CD) \arrow[dd, "{\alpha'_{AD,BD,CD}}"] \\
                                                                                                          &  & \text{(IV)}                                                               &  &                                                             \\
((A \oplus B) \oplus C)D \arrow[rr, "{\delta^{\#}_{A\oplus B,C,D}}"']                                     &  & (A \oplus B)D \oplus CD \arrow[rr, "{\delta^{\#}_{A,B,D}\oplus 1_{CD}}"'] &  & (AD \oplus BD) \oplus CD                                   
\end{tikzcd}

  \end{subfigure}
  \begin{subfigure}{0.99\textwidth}
    \centering

\begin{tikzcd}
A (B \oplus (C \oplus D)) \arrow[dd, "{1_A \alpha'_{B,C,D}}"'] \arrow[rr, "{\delta_{A,B,C\oplus D}}"] &  & AB \oplus A(C\oplus D) \arrow[rr, "{1_{AB} \oplus \delta_{A,C,D}}"] &  & AB\oplus (AC\oplus AD) \arrow[dd, "{\alpha'_{AB,AC,AD}}"] \\
                                                                                                      &  & \text{(V)}                                                          &  &                                                           \\
A((B\oplus C)\oplus D) \arrow[rr, "{\delta_{A, B\oplus C,D}}"']                                       &  & A(B\oplus C)\oplus AD \arrow[rr, "{\delta_{A,B,C}\oplus 1_{AD}}"']  &  & (AB \oplus AC)\oplus AD                                  
\end{tikzcd}

  \end{subfigure}
  \begin{subfigure}{0.99\textwidth}
    \centering
\begin{tikzcd}
A(B(C\oplus D)) \arrow[dd, "{\alpha_{A,B,C\oplus D}}"'] \arrow[rr, "{1_A \delta_{B,C,D}}"] &  & A(BC \oplus BD) \arrow[rr, "{\delta_{A,BC,BD}}"] &  & A(BC)\oplus A(BD) \arrow[dd, "{\alpha_{A,B,C}\oplus \alpha_{A,B,D}}"] \\
                                                                                           &  & \text{(VI)}                                      &  &                                                                       \\
(AB)(C\oplus D) \arrow[rrrr, "{\delta_{AB,C,D}}"']                                         &  &                                                  &  & (AB)C\oplus (AB)D                                                    
\end{tikzcd}
    
  \end{subfigure}
  \begin{subfigure}{0.99\textwidth}
    \centering

\begin{tikzcd}
(A \oplus B)(CD) \arrow[rrrr, "{\delta^{\#}_{A,B,CD}}"] \arrow[dd, "{\alpha_{A\oplus B, C, D}}"'] &  &                                                        &  & A(CD) \oplus B(CD) \arrow[dd, "{\alpha_{A,C,D} \alpha_{B,C,D}}"] \\
                                                                                                  &  & \text{(VII)}                                           &  &                                                                  \\
((A\oplus B)C)D \arrow[rr, "{\delta^{\#}_{A,B,C} 1_D}"']                                          &  & (AC \oplus BC)D \arrow[rr, "{\delta^{\#}_{AC,BC,D}}"'] &  & (AC)D \oplus (BC)D                                              
\end{tikzcd}

  \end{subfigure}
  \begin{subfigure}{0.99\textwidth}
    \centering

\begin{tikzcd}
A((B\oplus C)D) \arrow[rr, "{1_A\delta^{\#}_{B,C,D}}"] \arrow[dd, "{\alpha_{A,B\oplus C,D}}"'] &  & A(BD \oplus CD) \arrow[rr, "{\delta_{A,BD,CD}}"]       &  & A(BD) \oplus A(CD) \arrow[dd, "{\alpha_{A,B,D} \alpha_{A,C,D}}"] \\
                                                                                               &  & \text{(VIII)}                                          &  &                                                                  \\
(A(B\oplus C))D \arrow[rr, "{\delta_{A,B,C} 1_D}"']                                            &  & (AB \oplus AC)D \arrow[rr, "{\delta^{\#}_{AB,AC,D}}"'] &  & (AB)D \oplus (AC)D                                              
\end{tikzcd}

  \end{subfigure}
\end{figure}
\begin{figure}[H]
  \ContinuedFloat
  \centering
  \begin{subfigure}{0.99\textwidth}
    \centering

\begin{tikzcd}
A(C\oplus D) \oplus B(C\oplus D) \arrow[rr, "{\delta_{A,C,D} \delta_{B,C,D}}"]                               &             & (AC \oplus AD) \oplus (BC \oplus BD) \arrow[d, "{\alpha'_{AC\oplus AD, BC, BD}}"]               \\
                                                                                                             &             & ((AC \oplus AD) \oplus BC)\oplus BD \arrow[d, "{\alpha'^{-1}_{AC,AD,BC} \oplus 1_{BD}}"]        \\
                                                                                                             &             & (AC\oplus (AD\oplus BC))\oplus BD \arrow[dd, "{(1_{AC} \oplus \gamma'_{AD,BC}) \oplus 1_{BD}}"] \\
(A\oplus B)(C\oplus D) \arrow[uuu, "{\delta^{\#}_{A,B,C\oplus D}}"] \arrow[ddd, "{\delta_{A\oplus B,C,D}}"'] & \text{(IX)} &                                                                                                 \\
                                                                                                             &             & (AC \oplus (BC \oplus AD))\oplus BD                                                             \\
                                                                                                             &             & ((AC \oplus BC) \oplus AD) \oplus BD \arrow[u, "{\alpha'^{-1}_{AC, BC, AD} \oplus 1_{BD}}"']    \\
(A\oplus B)C \oplus (A\oplus B)D \arrow[rr, "{\delta^{\#}_{A,B,C} \delta^{\#}_{A,B,D}}"']                    &             & (AC \oplus BC) \oplus (AD \oplus BD) \arrow[u, "{\alpha'_{AC\oplus BC, AD, BD}}"']             
\end{tikzcd}
  \end{subfigure}
  \begin{subfigure}{0.49\textwidth}

    $$\text{(X): } \lambda^*_O = \rho^*_O : O \otimes O \rightarrow O$$
  \end{subfigure}
  \begin{subfigure}{0.49\textwidth}
    \centering

\begin{tikzcd}
O(A \oplus B) \arrow[dd, "\lambda^*_{A \oplus B}"'] \arrow[rr, "{\delta_{O,A,B}}"] &             & OA \oplus OB \arrow[dd, "\lambda^*_A \oplus \lambda^*_B"] \\
                                                                                   & \text{(XI)} &                                                           \\
O                                                                                  &             & O \oplus O \arrow[ll, "\lambda'_O"]                      
\end{tikzcd}

  \end{subfigure}
  \begin{subfigure}{0.49\textwidth}
    \centering


\begin{tikzcd}
(A\oplus B)O \arrow[dd, "\rho^*_{A \oplus B}"'] \arrow[rr, "{\delta^{\#}_{A,B,O}}"] &              & AO \oplus BO \arrow[dd, "\rho^*_A \oplus \rho^*_B"] \\
                                                                                    & \text{(XII)} &                                                     \\
O                                                                                   &              & O \oplus O \arrow[ll, "\lambda'_O"]                
\end{tikzcd}
    
  \end{subfigure}
  \begin{subfigure}{0.49\textwidth}


    $$\text{(XIII): } \rho_O = \lambda^*_O : O I \rightarrow O$$

  \end{subfigure}
  \begin{subfigure}{0.49\textwidth}

    $$\text{(XIV): } \lambda_O = \rho^*_O : I O \rightarrow O$$

  \end{subfigure}
  \begin{subfigure}{0.49\textwidth}
    \centering
    
    

\begin{tikzcd}
O(AB) \arrow[rr, "{\alpha_{O,A,B}}"] \arrow[dd, "\lambda^*_{AB}"'] &              & (OA)B \arrow[dd, "\lambda^*_A 1_B"] \\
                                                                   & \text{(XVI)} &                                     \\
O                                                                  &              & OB \arrow[ll, "\lambda^*_B"]       
\end{tikzcd}

  \end{subfigure}
  \begin{subfigure}{0.49\textwidth}
\centering

\begin{tikzcd}
A(OB) \arrow[rr, "{\alpha_{A,O,B}}"] \arrow[dd, "1_A \lambda^*_B"'] &               & (AO)B \arrow[dd, "\rho^*_A 1_B"] \\
                                                                    & \text{(XVII)} &                                  \\
AO \arrow[rd, "\rho^*_A"']                                          &               & OB \arrow[ld, "\lambda^*_B"]     \\
                                                                    & O             &                                 
\end{tikzcd}
  \end{subfigure}
  \begin{subfigure}{0.49\textwidth}
\centering

\begin{tikzcd}
A(BO) \arrow[rr, "1_A \rho^*_B"] \arrow[dd, "{\alpha_{A,B,O}}"'] &                & AO \arrow[dd, "\rho^*_A"] \\
                                                                 & \text{(XVIII)} &                           \\
(AB)O \arrow[rr, "\rho^*_{AB}"']                                 &                & O                        
\end{tikzcd}
  \end{subfigure}
\end{figure}
\begin{figure}
  \ContinuedFloat
  \centering
  \begin{subfigure}{0.49\textwidth}
\centering

\begin{tikzcd}[column sep=small]
A(O \oplus B) \arrow[rr, "{\delta_{A,O,B}}"] \arrow[dd, "1_A \lambda'_B"'] &              & AO \oplus AB \arrow[dd, "\rho^*_A \oplus 1_{BA}"] \\
                                                                           & \text{(XIX)} &                                                   \\
AB                                                                         &              & O \oplus AB \arrow[ll, "\lambda'_{AB}"]          
\end{tikzcd}
  \end{subfigure}
  \begin{subfigure}{0.49\textwidth}
\centering

\begin{tikzcd}[column sep=small]
(O \oplus B) A \arrow[rr, "{\delta^{\#}_{O,B,A}}"] \arrow[dd, "\lambda'_B 1_A"'] &             & OA \oplus BA \arrow[dd, "\lambda^*_A \oplus 1_{BA}"] \\
                                                                                 & \text{(XX)} &                                                      \\
BA                                                                               &             & O \oplus BA \arrow[ll, "\lambda'_{BA}"]             
\end{tikzcd}
  \end{subfigure}
  \begin{subfigure}{0.49\textwidth}
    \centering

\begin{tikzcd}[column sep=small]
A(B \oplus O) \arrow[rr, "{\delta_{A,B,O}}"] \arrow[dd, "1_A \rho'_B"'] &              & AB \oplus AO \arrow[dd, "1_{AB} \oplus \rho^*_A"] \\
                                                                        & \text{(XXI)} &                                                   \\
AB                                                                      &              & AB \oplus O \arrow[ll, "\rho'_{AB}"]             
\end{tikzcd}

  \end{subfigure}
  \begin{subfigure}{0.49\textwidth}
\centering

\begin{tikzcd}[column sep=small]
(A \oplus O)B \arrow[rr, "{\delta^{\#}_{A,O,B}}"] \arrow[dd, "\rho'_A 1_B"'] &               & AB \oplus OB \arrow[dd, "1_{AB} \oplus \lambda^*_B"] \\
                                                                             & \text{(XXII)} &                                                      \\
AB                                                                           &               & AB \oplus O \arrow[ll, "\rho'_{AB}"]                
\end{tikzcd}

  \end{subfigure}
  \begin{subfigure}{0.55\textwidth}


\begin{tikzcd}
I(A \oplus B) \arrow[rr, "{\delta_{I,A,B}}"] \arrow[rdd, "\lambda_{A \oplus B}"'] &                & IA \oplus IB \arrow[ldd, "\lambda_A \oplus \lambda_B"] \\
                                                                                  & \text{(XXIII)} &                                                        \\
                                                                                  & A \oplus B     &                                                       
\end{tikzcd}

  \end{subfigure}
  \begin{subfigure}{0.55\textwidth}


\begin{tikzcd}
(A \oplus B)I \arrow[rr, "{\delta^{\#}_{A,B,I}}"] \arrow[rdd, "\rho_{A \oplus B}"'] &               & AI \oplus BI \arrow[ldd, "\rho_A \oplus \rho_B"] \\
                                                                                    & \text{(XXIV)} &                                                  \\
                                                                                    & A \oplus B    &                                                 
\end{tikzcd}

  \end{subfigure}
  \label{fig:coherence-axioms}
  \caption{Coherence axioms}
\end{figure}

\section{Axioms of a bimonoidal functor} \label{app:bimonoidal-functor}

\begin{definition} \cite[Def. 2.5.1]{elgueta2020groupoid}
A {\em (strong) bimonoidal functor} between bimonoidal categories $(\C, \cdot^\C, I^\C, \oplus^\C, O^\C )$ and $(\D, \cdot^\D, I^\D, \oplus^\D, O^\D)$ is a functor $F:\C\to\D$, isomorphisms 
$$\varepsilon^\cdot:I^\D \to F(I^\C)\hspace*{1cm}\varepsilon^\oplus:O^\D \to F(O^\C)$$
and natural isomorphisms
$$\mu_{A,B}^\cdot:F(A) \cdot^\D F(B) \to F(A\cdot^\C B)\hspace*{1cm}
\mu_{A,B}^\oplus:F(A) \oplus^\D F(B) \to F(A\oplus^\C B)$$
So that 
$(F,\eta^\cdot,\mu^\cdot)$ is a strong monoidal functor from $(\C, \cdot^\C, I^\C)$ to $(\D, \cdot^\D, I^\D)$,
$(F,\eta^\oplus,\mu^\oplus)$ is a strong monoidal functor from $(\C, \oplus^\C, O^\C)$ to $(\D, \oplus^\D, O^\D)$, where the following diagrams also commute:

$$
\xymatrixcolsep{15mm}
\mbox{\footnotesize
\xymatrix{
FA\cdot^\D F(B \oplus^\C C) \ar[d]_{\mu_{A, B\oplus^\C C}^\cdot} \ar[r]^{1_{FA} \cdot^\D \mu_{B,C}^\oplus} & FA\cdot^\D (FB \oplus^\C FC) \ar[r]^{\delta_{FA,FB,FC}^\D} & FA \cdot^\D FB \oplus^\D FA\cdot^\D FC  \ar[d]^{\mu_{A,B}^\cdot \oplus^\D \mu_{A,C}^\cdot}\\
F(A\cdot^\C (B\oplus^\C D)) \ar[r]_{F(\delta_{A,B,C}^\C)} & F(A\cdot^\C B \oplus^\C A\cdot^\C D) \ar[r]_{\mu_{A\cdot^\C B, A\cdot^{\D} C}^\oplus} & F(A \cdot^\C B) \oplus^\D F(A \cdot^\C C)
}}
$$

$$
\xymatrixcolsep{20mm}
\xymatrix{
FA\cdot^\D FO^\C\ar[d]_{\mu_{A,O^C}^\cdot}  \ar[r]^{1_{FA} \cdot^\D \varepsilon^\oplus} & FA\cdot^\D O^\D \ar[r]^{\hspace*{.6cm}\rho_O^{*,\D}} & O^\D \ar@{=}[d]\\
F(A\cdot^\C O^\C) \ar[r]_{\rho_O^{*,\C}} & FO^C \ar[r]_{\varepsilon^\oplus} & O^\D \\ 
}
$$

$$
\xymatrixcolsep{15mm}
\mbox{\footnotesize
\xymatrix{
F(A\oplus^\C B) \cdot^\D  FC \ar[d]_{\mu_{A\oplus^\C B, C}^\cdot} \ar[r]^{ \mu_{A,B}^\oplus \cdot^\D1_{FC} }  & (FA \oplus^\C FB) \cdot^\D FC \ar[r]^{\delta_{FA,FB,FC}^{\#,\D}} & FA \cdot^\D FC \oplus^\D FB\cdot^\D FC  \ar[d]^{\mu_{A,C}^\cdot \oplus^\D \mu_{B,C}^\cdot}\\
F((A\oplus^\C B)\cdot^\C C))\ar[r]_{F(\delta_{A,B,C}^{\#,\C})} & F(A\cdot^\C C \oplus^\C B\cdot^\C C) \ar[r]_{\mu_{A\cdot^\C C, B\cdot^{\D} C}^\oplus} & F(A \cdot^\C C) \oplus^\D F(B \cdot^\C C)
}}
$$

$$
\xymatrixcolsep{20mm}
\xymatrix{
FO^\C \cdot^\D FA \ar[d]_{\mu_{O^C,A}^\cdot}  \ar[r]^{ \varepsilon^\oplus \cdot^\D 1_{FA} } & O^\D \cdot^\D  FA \ar[r]^{\hspace*{.6cm}\lambda_O^{*,\D}} & O^\D \ar@{=}[d]\\
F(O^\C \cdot^\C  A) \ar[r]_{\lambda_O^{*,\C}} & FO^C \ar[r]_{\varepsilon^\oplus} & O^\D \\ 
}
$$

\end{definition}

\end{document}